\documentclass{siamart0516mod}

\usepackage{amssymb,amsfonts,amsmath}

\usepackage[mathscr]{eucal}

\usepackage{stmaryrd}

\usepackage{amsbsy}

\usepackage{bm}

\usepackage{braket}

\usepackage{multirow}

\usepackage{tabularx}
\newcolumntype{Y}{>{\raggedright\arraybackslash}X}

\usepackage{xparse}

\usepackage{xfrac}

\usepackage{enumitem}

\usepackage{environ}

\usepackage{tikz}
\usetikzlibrary{3d}
\usetikzlibrary{patterns}
\usetikzlibrary{calc}
\usetikzlibrary{arrows}

\usepackage{pgfplots}

\usepackage[font=footnotesize,justification=Centering,singlelinecheck=false]{subfig}

\usepackage{cleveref}

\pdfpxdimen=\dimexpr 1 in/96\relax

\newenvironment{inlinemath}{$}{$}

\NewDocumentCommand \qtext {m} {\quad\text{#1}\quad}

\NewDocumentCommand \Real {} {\mathbb{R}}

\NewDocumentCommand \Natural {} {\mathbb{N}}

\NewDocumentCommand \T { O{} m } {\boldsymbol{#1\mathscr{\MakeUppercase{#2}}}}
\NewDocumentCommand \Mx { O{} m } {{\bm{#1\mathbf{\MakeUppercase{#2}}}}} %
\NewDocumentCommand \V { O{} m } {{\bm{#1\mathbf{\MakeLowercase{#2}}}}}

\NewDocumentCommand \X { } {\T{X}}

\NewDocumentCommand \Xk { O{k} } {\Mx{X}_{(#1)}}

\NewDocumentCommand \Xv { } {\V{x}}

\DeclareDocumentCommand \xi { s }
{
  \IfBooleanTF{#1}
  {x(i_1,i_2,\dots,i_d)}
  {x_{i}}
}

\NewDocumentCommand \M {} {\T{M}}

\NewDocumentCommand \Mk { O{k} } {\Mx{M}_{(#1)}}

\NewDocumentCommand \Mv { } {\V{m}}

\DeclareDocumentCommand \mi { s }
{
  \IfBooleanTF{#1}
  {m(i_1,i_2\dots,i_d)}
  {m_{i}}
}

\NewDocumentCommand \W {} {\T{W}}

\NewDocumentCommand \Wk {} {\Mx{W}_{(k)}}

\DeclareDocumentCommand \wi { s }
{
  \IfBooleanTF{#1}
  {w(i_1,i_2,\dots,i_d)}
  {w_{i}}
}

\NewDocumentCommand \Y {} {\T{Y}}

\NewDocumentCommand \Yk { O{k} } {\Mx{Y}_{(#1)}}

\NewDocumentCommand \Yv { } {\V{y}}

\DeclareDocumentCommand \yi { s }
{
  \IfBooleanTF{#1}
  {y(i_1,i_2,\dots,i_d)}
  {y_{i}}
}

\NewDocumentCommand \Ak { G{k} t' t"  } { \Mx{A}_{#1}\IfBooleanTF{#2}{^{\intercal}}{}\IfBooleanTF{#3}{^{\phantom{\intercal}}}{} }

\NewDocumentCommand \Akset { } {\set{\Ak | k=1,\dots,d}}

\NewDocumentCommand \AkAkt { G{k} } {\Ak{#1}'\Ak{#1}"}
\NewDocumentCommand \ZkZkt { G{k} } {\Zk{#1}'\Zk{#1}"}

\NewDocumentCommand \Ake { G{k} G{i} G{j} } {
  a_{#1}(#2_{#1},#3)
}

\NewDocumentCommand \Gk { G{k}  } { \Mx{G}_{#1} }
\NewDocumentCommand \Gkh { G{k}  } { \Mx[\hat]{G}_{#1} }
\NewDocumentCommand \Gkset { } {\set{\Gk | k=1,\dots,d}}

\NewDocumentCommand \lvec {} {\V{\lambda}}

\NewDocumentCommand \KT { s } {
  \llbracket
  \IfBooleanTF{#1}{\lvec;}{}
  \Ak{1}, \Ak{2}, \dots,  \Ak{d} \rrbracket
}

\NewDocumentCommand \avec {} {\V{a}}

\NewDocumentCommand \gvec {} {\V{g}}

\NewDocumentCommand \Zk { G{k} t' t"} {\Mx{Z}_{#1}\IfBooleanTF{#2}{^{\intercal}}{}%
  \IfBooleanTF{#3}{^{\phantom{\intercal}}}{}}

\NewDocumentCommand \zvec { } {\Mx{Z}}

\NewDocumentCommand \A { } {\Mx{A}}
\NewDocumentCommand \B {t' } {\Mx{B}\IfBooleanTF{#1}{^{\intercal}}{}}
\NewDocumentCommand \I {} {\mathcal{I}}

\NewDocumentCommand \nnset {s m}
{
  \IfBooleanTF{#1}
  {\set{1,2,\dots, #2}}
  {\set{1, \dots, #2}}
}

\NewDocumentCommand \FPD { s m m } {
  \IfBooleanTF{#1}
  {\partial #2 / \partial #3}
  {\frac{\partial #2}{\partial #3}}
}

\NewDocumentCommand \pdf {m m} {p(#1\,\vert\,#2)}

\DeclareMathOperator{\diag}{diag}

\NewDocumentCommand{\vc}{}{\textsc{vec}}

\NewDocumentCommand{\expt}{m}{\mathbb{E}[#1]}

\NewDocumentCommand{\Bern}{}{{\rm{Bernoulli}}}
\NewDocumentCommand{\ifrac}{s m m}{#2 \, \big/ \, \IfBooleanTF{#1}{#3}{(#3)}}
\NewDocumentCommand{\logfrac}{m m}{\log \bigl(\, #1 \, \big/ \, (#2) \, \bigr) }

\NewDocumentCommand \bk { G{k}  } { \V{u}_{#1} }
\NewDocumentCommand \ck { G{k}  } { \V{v}_{#1} }

\NewDocumentCommand \Bk {G{k}} {\Mx{U}_{#1}}
\NewDocumentCommand \Ck {G{k}} {\Mx{V}_{#1}}

\NewDocumentCommand \dOmega {} {\delta_{i \in \Omega}}

\NewDocumentCommand \MouseCaption { } {Components ordered by size (top
  to bottom). %
  Example neurons (26, 62, 82, 117, 154,
  176, 212, 249, 273) highlighted in red. %
  Trial symbols are coded by conditions: color indicates
  turn and  filled indicates a reward.
  The rule changes are denoted by vertical dotted lines.
  Observe that some factors split the trials by turn (green versus orange) and others split by reward (open versus filled), even though the tensor decomposition has no knowledge of the trial conditions.}

\definecolor{constant}{rgb}{0,0,0.8}
\definecolor{fix}{rgb}{0.8,0,0}
\NewDocumentCommand \addfix {} {\textcolor{fix}{+\epsilon}}
\NewDocumentCommand \rge {} {\textcolor{fix}{\;\geq\;}}

\Crefname{ALC@unique}{Line}{Lines}

\title{Generalized Canonical Polyadic Tensor Decomposition%
  \thanks{Sandia National Laboratories is a multimission laboratory
    managed and operated by National Technology and Engineering
    Solutions of Sandia, LLC., a wholly owned subsidiary of Honeywell
    International, Inc., for the U.S. Department of Energy's National
    Nuclear Security Administration under contract DE-NA-0003525.
    This paper describes objective technical results and analysis. Any
    subjective views or opinions that might be expressed in the paper
    do not necessarily represent the views of the U.S. Department of
    Energy or the United States Government.}}

\author{%
  David Hong%
  \thanks{University of Michigan, Ann Arbor, MI (\email{dahong@umich.edu})}
  \and Tamara~G.~Kolda%
  \thanks{Sandia National Laboratories, Livermore, CA (\email{tgkolda@sandia.gov})}
  \and Jed A. Duersch%
  \thanks{Sandia National Laboratories, Livermore, CA (\email{jaduers@sandia.gov})}
 }

\headers{Generalized CP Tensor Decomposition}
{David Hong, Tamara G. Kolda, and Jed A. Duersch}

\begin{document}

\maketitle

\begin{abstract}
Tensor decomposition is a fundamental unsupervised machine learning method in data science, with applications including network analysis and sensor data processing. This work develops a generalized canonical polyadic (GCP) low-rank tensor decomposition that allows other loss functions besides squared error. For instance, we can use logistic loss or Kullback-Leibler divergence, enabling tensor decomposition for binary or count data. We present a variety of statistically-motivated loss functions for various scenarios. We provide a generalized framework for computing gradients and handling missing data that enables the use of standard optimization methods for fitting the model. We demonstrate the flexibility of GCP on several real-world examples including interactions in a social network, neural activity in a mouse, and monthly rainfall measurements in India.
\end{abstract}


\section{Introduction}
\label{sec:introduction}

Many data sets are naturally represented as higher-order tensors.
The CANDECOMP/PARAFAC or canonical polyadic (CP) tensor decomposition
builds a low-rank tensor decomposition model and
is a standard tool for unsupervised multiway data analysis \cite{Hi27a,CaCh70,Ha70,KoBa09}.
The CP decomposition is analogous to principal component analysis or the singular value decomposition for two-way data.
Structural features in the dataset are represented as rank-1 tensors, which reduces the size and complexity of the data.
This form of dimensionality reduction has many applications including
data decomposition into explanatory factors, dimensionality reduction, filling in missing data, and data compression.
It has been used to analyze multiway data sets in a variety of domains including neuroscience \cite{AcBiBiBr07,WiKiWaVy18,CoLiKuGo15}, chemistry \cite{MuStGrBr13,JaCaYa14}, cybersecurity \cite{MaGuFa11}, network analysis and link prediction \cite{KoBa06,OpPa09,DuKoAc11}, machine learning \cite{AcYe09,BeGaMo09,Re12}, hyperspectral imaging \cite{ZhWaPlPa08,FaHeLi17}, function approximation \cite{BeMo02,BeMo05,HaKh07,ReDoBe16}, and so on.
In this manuscript, we consider generalizing the \emph{loss function}
for determining the low-rank model.

\begin{figure}[th]
  \centering
  %
\newcommand{\cpfactor}{%
\begin{tikzpicture}
\draw (-0.125,-1) -- (-0.125,1) -- (0.125,1) -- (0.125,-1) -- cycle;
\draw (0.25,1) -- (1.5,1) -- (1.5,0.75) -- (0.25,0.75) -- cycle;
\draw (-0.125,1.125) -- (0.125,1.125) -- (0.75,1.75) -- (0.5,1.75) -- cycle;
\end{tikzpicture}%
}
\newcommand{\tensorblock}{%
\begin{tikzpicture}
\draw (-0.125,-1) -- (-0.125,1) -- (1.5,1) -- (1.5,-1) -- cycle;
\draw (-0.125,1) -- (1.5,1) -- (2.25,1.75) -- (0.625,1.75) -- cycle;
\draw (1.5,1) -- (2.25,1.75) -- (2.25,-0.25) -- (1.5,-1) -- cycle;
\end{tikzpicture}%
}

\begin{tikzpicture}
\node at (0,-0.04) {\tensorblock};
\node at (1.50,0) {\Large $=$};
\node at (2.75,0) {\cpfactor};
\node at (4.00,0) {\Large $+$};
\node at (5.25,0) {\cpfactor};
\node at (6.50,0) {\Large $+$};
\node at (7.25,0) {\Large $\cdots$};
\node at (8.00,0) {\Large $+$};
\node at (9.25,0) {\cpfactor};
\end{tikzpicture}
  \caption{Illustration of CP-structured tensor. The tensor is the sum
    of $r$ components, and each component is the outer product of $d$
    vectors, also known as a rank-1 tensor (here we show $d=3$).
    The rank of such a tensor that has $r$ components is bounded
    above by $r$, so it is low-rank if $r$ is small.}
  \label{fig:cp_structure}
\end{figure}
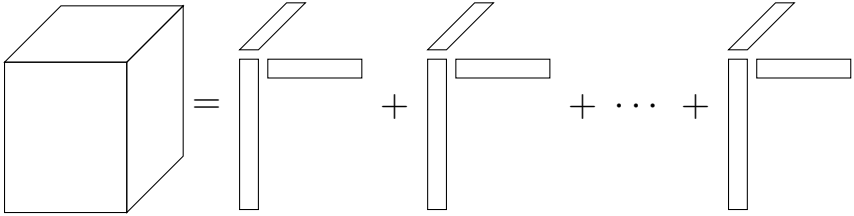

Given a $d$-way data tensor $\X$ of size $n_1 \times n_2 \times \cdots \times n_d$,
we propose a generalized CP (GCP) decomposition
that approximates $\X$ as measured by the sum of elementwise losses specified by a generic function
$f:\Real \otimes \Real \rightarrow \Real$, i.e.,
\begin{equation}\label{eq:gcp-loss}
  \min F(\M; \X) \equiv
  \sum_{i_1=1}^{n_1} \sum_{i_2=1}^{n_2} \cdots \sum_{i_d=1}^{n_d}
  f(\xi,\mi)
  \qtext{subject to}
  \text{$\M$ is low rank}.
\end{equation}
Here, $\M$ is a low-rank model tensor that has CP structure, as
illustrated in \cref{fig:cp_structure}.
We use the shorthand $\xi = \xi*$ and $\mi = \mi*$.
For the usual CP decomposition, the elementwise loss is
\begin{inlinemath}
  f(\xi,\mi) = (\xi-\mi)^2.
\end{inlinemath}
While this loss function is suitable for many situations,
it implicitly assumes the data is normally distributed.
Many datasets of interest, however, do not satisfy this hidden assumption.
Such data can be nonnegative, discrete, or boolean.

Our goal in this paper is to develop a general framework for fitting GCP
models with generic loss functions, enabling the user to adapt the model
to the nature of the data.  For example,
we later see that a natural elementwise loss function for binary tensors,
which have all entries in $\set{0,1}$,
is $f(\xi,\mi) = \log (\mi+1) - \xi \log \mi$.
We show that the GCP gradient has an elegant form that uses the
same computational kernels as the standard CP gradient.
We also cover the case of incomplete tensors, where some data is missing due
to either collection issues or an inability to make
measurements in some scenarios. This is a common issue for real-world datasets,
and it can be easily handled in the GCP framework.

\subsection{Connections to prior work}
\label{sec:conn-prior-work}

Applications of the CP tensor decomposition date back to the 1970 work
of Carrol and Chang \cite{CaCh70} and Harshman \cite{Ha70}, though its
mathematical origins date back to Hitchcock in 1927 \cite{Hi27a}.
Many surveys exist on CP and its application; see, for instance,
\cite{Br97,AcYe09,KoBa09,PaKaFaSi13}.

Our proposed GCP framework uses so-called direct or all-at-once optimization, in contrast to the alternating approach
that is popular for computing CP known as CP-ALS. The direct optimization approach has been considered for
CP by Acar, Dunlavy, and Kolda \cite{DuKoAc11} and Phan, Tichavsk{\'y}, and Cichocki \cite{PhTiCi13a}.
The later case showed that the Hessians have special structure, and similar structure applies in the case of GCP though we do not discuss it here.
The GCP framework can incorporate many of the computational improvements for CP, such as tree-based MTTKRP computations \cite{PhTiCi13}
and ADMM for constraints \cite{HuSiLi15}.
Our approach for handling missing data is essentially the same as that
proposed for standard CP by Acar, Dunlavy, Kolda, and M{\o}rup
\cite{AcDuKoMo11};
the primary difference is that we now have a more elegant and general theory for the
derivatives.

There have been a wide variety of papers that have considered
alternative loss functions, so here we mention some of the most relevant.
The famous nonnegative matrix factorization paper of Lee and Seung \cite{LeSe99}
considered KL divergence in the matrix case,
and Welling and Weber \cite{WeWe01} and others \cite{ShHa05,ChKo12,HaPlKo15}
considered it in the tensor case.
This equates to Poisson
with identity link \cref{eq:poisson} in our framework.
Cichocki, Zdunek, Choi, Plemmons, and Amari \cite{CiZdChPl07} have
considered loss functions based on alpha- and beta-divergences for
nonnegative CP \cite{CiPh09}.  These fit into the GCP framework, and
we explicitly discuss the case of beta divergence.

To the best of our knowledge, no general loss function frameworks have been proposed in the tensor case, but several have been proposed in the matrix case.
Collins, Dasgupta, and Schapire \cite{CoDaSc02}
developed a generalized version of matrix principal component analysis (PCA) based on
loss functions from the exponential family (Gaussian, Poisson with
exponential link, Bernoulli with logit link).
Gordon \cite{Go03} considers a ``Generalized$^2$ Linear$^2$ Model'' for matrix factorization that allows different loss functions and nonlinear relationships between the factors and the low-rank approximation.
Udell et~al.~\cite{UdHoZaBo16} develop a general framework for matrix factorization
that allows for the loss function to be different for each column;
several of their proposed loss functions overlap with ours
(e.g., their ``Poisson PCA'' is equivalent to Poisson with the log link).

\subsection{Contributions}
\label{sec:contributions}

We develop the GCP framework for computing the CP tensor decomposition with
an arbitrary elementwise loss function.
\begin{itemize}
\item The main difference between GCP and standard CP is the choice of
  loss function, so we discuss loss function choices and their
  statistical connections in \cref{sec:loss-functions}.
\item We describe fitting the GCP model in
  \cref{sec:fitting-gcp-small}.  We derive the gradient for GCP with
  respect to the model components, along with a straightforward way of
  handling missing data. We explain how to add regularization and use
  a standard optimization method. %
\item In \cref{sec:experimental-results}, we demonstrate the
  flexibility of GCP on several real-world examples with corresponding
  applications including inference of missing entries,
  and unsupervised pattern extraction over a
  variety of data types.
\end{itemize}

\section{Background and notation}
\label{sec:background-notation}

Before we continue, we establish some basic tensor notation and concepts; see Kolda and Bader \cite{KoBa09} for a full review.

A boldface uppercase letter in Euler font denotes a tensor, e.g., $\X$.
The number of ways or dimensions of the tensor is called the \emph{order}.
Each way is referred to as a \emph{mode}.
A boldface uppercase letter represents a matrix, e.g., $\mathbf{A}$.
A boldface lowercase letter represents a vector, e.g., $\mathbf{v}$.
A lowercase letter represents a scalar, e.g., $x$.

The Hadamard (elementwise) product of two same-sized matrices
$\mathbf{A}, \mathbf{B} \in \Real^{m \times n}$
is denoted by $\mathbf{A} \ast \mathbf{B}  \in \Real^{m \times n}$.
The {Khatri-Rao product} of two matrices
$\mathbf{A} \in \Real^{m \times p}$ and $\mathbf{B} \in \Real^{n \times p}$
is the columnwise Kronecker product, i.e.,
\begin{equation}
\label{eq:KR}
  \mathbf{A} \odot \mathbf{B} =
  \begin{bmatrix}
    a_{11} b_{11} & a_{12} b_{12} & \cdots & a_{1p} b_{1p} \\
    a_{11} b_{21} & a_{12} b_{22} & \cdots & a_{1p} b_{2p} \\
    \vdots & \vdots & \ddots & \vdots \\
    a_{m1} b_{n1} & a_{m2} b_{n2} & \cdots & a_{mp} b_{np} \\
  \end{bmatrix}
  \in \Real^{mn \times p}.
\end{equation}

\subsection{General tensor notation}
\label{sec:tensor-notation}

In the remainder of the paper, we assume all tensors are real-valued $d$-way
arrays of size $n_1 \times n_2 \times \cdots \times n_d$.
We define $n$ and $\bar n$ to be the geometric and arithmetic means of the sizes, i.e.,
\begin{equation}
\label{eq:n}
  n = \sqrt[d]{ \prod_{k=1}^d n_k }
  \qtext{and}
  \bar n = \frac{1}{d} \sum_{k=1}^d n_k.
\end{equation}
In this way, $n^d$ is the total number of elements in the tensor
and $d \bar n$ is the sum of the sizes of all the modes.
As shown above, modes are typically indexed by $k \in \nnset{d}$.

Tensors are indexed using $i$ as shorthand for the \emph{multiindex} $(i_1,i_2,\dots,i_d)$, so that $\xi \equiv \xi*$.
We let $\I$ denote the set of all possible indices, i.e.,
\begin{equation}
\label{eq:set-I}
\I \equiv \nnset{n_1} \otimes \nnset{n_2} \otimes \cdots \otimes \nnset{n_d}.
\end{equation}
It may be the case that some entries of $\X$ are \emph{missing}, i.e.,
were not observed due to measurement problems. We let $\Omega \subseteq \I$ denote the set
of \emph{observed} entries, and then $\I \setminus \Omega$ is the set of missing entries.

The \emph{vectorization} of $\X$ rearranges its elements into a vector of size $n^d$ and is denoted by $\Xv$.
Tensor element $\xi*$ is mapped to $x(i')$ in $\Xv$ where the \emph{linear index} $i' \in \nnset{n^d}$ is given by
$i' = 1 + \sum_{k = 1}^{d} (i_{k} - 1) n'_k$ with $n'_1 = 1$ and $n'_k = \prod_{\ell=1}^{k-1} n_{\ell}$ otherwise.

The mode-$k$ \emph{unfolding} or \emph{matricization}
of $\X$ rearranges its elements into a matrix of size
$n_k \times (n^d/n_k)$ and is denoted as $\Xk$, where the subscript
indicates the mode of the unfolding.
Element $(i_1,\dots,i_d) \in \I$ maps to matrix entry $(i_k,i_k')$ where
\begin{equation}
\label{eq:unfolding}
i_k' = 1 + \sum_{\ell = 1}^{k-1} (i_{\ell} - 1) n'_{\ell}
             + \sum_{\ell = k+1}^{d} (i_{\ell} - 1) (n'_{\ell} / n_k)
\end{equation}

\subsection{Kruskal tensor notation}
\label{sec:kruskal}
We assume the model tensor $\M$ in \cref{eq:gcp-loss} has low-rank
CP structure as illustrated in \cref{fig:cp_structure}.
Following Bader and Kolda \cite{BaKo07}, we refer to this type of tensor as a \emph{Kruskal tensor}.
Specifically, it is defined by a set of $d$ \emph{factor matrices}, $\Ak$ of size $n_k \times r$ for $k=1,\dots,d$,
such that
\begin{equation}
\label{eq:M-low-rank}
\mi \equiv \mi* = \sum_{j=1}^r a_1(i_1,j) a_2(i_2,j) \cdots a_d(i_d,j) \qquad\text{ for all } i \in \I.
\end{equation}
The number of columns $r$ is the same for all factor matrices and equal to the number of \emph{components} ($d$-way outer products) in the model.
In \cref{fig:cp_structure}, the $j$th component is the outer product of the $j$th column vectors of the factor matrices,
i.e., $\Ak{1}(:,j)$, $\Ak{2}(:,j)$, etc.
We denote~\cref{eq:M-low-rank} in shorthand as $\M = \KT$.
The mode-$k$ unfolding of a Kruskal tensor has a special form
that depends on the Khatri-Rao products of the factor matrices, i.e.,
\begin{equation}
  \label{eq:Zk}
  \Mk = \Ak" \Zk'
  \qtext{where}
  \Zk \equiv
  \Ak{d} \odot \cdots \odot \Ak{k+1} \odot
  \Ak{k-1} \odot \cdots \odot \Ak{1}
  .
\end{equation}

If $r$ is relatively small (e.g., $r \leq \mathcal{O}(n)$), then we say $\M$ has \emph{low rank}.
The advantage of finding a low-rank structure is that it is more parsimonious.
The model $\M$ has $n^d$ entries but the number of values to define it is only
\begin{displaymath}
  r \sum_{k=1}^d n_k = d r \bar n \ll n^d.
\end{displaymath}

It is sometimes convenient to normalize the columns of the factor
matrices and have an explicit weight for each component. For
clarity of presentation, we omit this from our main discussion but
do provide this alternative form and related results in \cref{sec:explicit-weights}.

\section{Choice of loss function}
\label{sec:loss-functions}

The difference between GCP and the standard CP formulation is
flexibility in the choice of loss function. In this section, we
motivate alternative loss functions by looking at the statistical
likelihood of a model for a given data tensor.

In statistical modeling, we often want to maximize the \emph{likelihood}
of a model that parameterizes the distribution; see, e.g., \cite[section 8.2.2]{HaTiFr09}.
We assume that we have a parameterized
probability density function (PDF) or probability mass function (PMF)
that gives the likelihood of each entry, i.e.,
\begin{displaymath}
  \xi \sim \pdf{\xi}{\theta_i} \qtext{where} \ell(\theta_i) = \mi,
\end{displaymath}
where $\xi$ is an observation of a random variable and $\ell(\cdot)$ is an invertible \emph{link function}
that connects the model parameter $\mi$ and the corresponding \emph{natural parameter} of the distribution, $\theta_i$.
The link function is oftentimes just the identity function,
but we show the utility of a nontrivial link function in \cref{sec:bernoulli-odds-link}.
Link functions are a common statistical concept and have been used for  generalized matrix factorizations~\cite{CoDaSc02,Go03}.

Our goal is to find the model $\M$ that is the \emph{maximum
  likelihood estimate} (MLE) across all entries.  Conditional
independence of observations\footnote{The independence is conditioned
  on $\M$. Although there are dependencies between the entries of $\M$
  since indeed the entire purpose of the GCP decomposition is to
  discover these dependencies, the observations themselves remain conditionally
  independent.}  means that
the %
overall likelihood is just the product of the likelihoods, so the MLE
is the solution to
\begin{equation}
  \label{eq:joint-prob}
  \max_{\M} \;
  L(\M;\X) \equiv
  \prod_{i \in \Omega} \pdf{\xi}{\theta_i}
  \qtext{with}
  \ell(\theta_i) = \mi \text{ for all } i \in \Omega.
\end{equation}
We are trying to estimate the  parameters $\theta_i$, but we only have \emph{one} observation per random variable $\xi$.
Nevertheless, we are able to make headway because of the low-rank structure of $\M$ and corresponding interdependences of the $\theta_i$'s.
Recall that we have $n^d$ observations but only $d r \bar n$ free variables.

For a variety of reasons, expression \cref{eq:joint-prob} is rather awkward.
Instead we take the negative logarithm to convert the product into a sum.
Since the log is monotonic, it does not change the maximizer.
Negation simply converts the maximization problem into a minimization problem which is common for optimization.
Eliminating $\theta_i$ as well, we arrive at the minimization problem
\begin{equation}
\label{eq:general-function}
\min F(\M;\X) \equiv \sum_{i \in \Omega} f(\xi,\mi)
\qtext{where} f(x,m) \equiv - \log \pdf{x}{\ell^{-1}(m)}.
\end{equation}
In the remainder of this section, we discuss how specific choices of distributions (and corresponding $p$'s) lead
naturally to specific choices for the elementwise loss function $f$.
Each distribution has its own standard notation for the generic parameter $\theta$,
e.g., the Poisson distribution in \cref{sec:poiss-ident-link} refers to its natural parameter as $\lambda$.
Although our focus here is on statistically-motivated choices for the loss function, other options
are possible as well. We mention two, the Huber loss and $\beta$-divergence, explicitly in \cref{sec:choos-loss-funct}.

\subsection{Gaussian distribution and the standard formulation}
\label{sec:gaussian}

In this subsection, we show that the standard squared error loss
function, $f(x,m) = (x-m)^2$, comes from an assumption that the data is Gaussian distributed.
A usual assumption is
that the data has low-rank structure but is contaminated by
``white noise,'' i.e.,
\begin{equation}\label{eq:white_noise}
  \xi = \mi + \epsilon_i \qtext{with} \epsilon_i \sim \mathcal{N}(0,\sigma)
  \qtext{for all} i \in \Omega.
\end{equation}
Here $\mathcal{N}(\mu,\sigma)$ denotes the normal or Gaussian distribution with mean $\mu$ and standard deviation $\sigma$.
We assume $\sigma$ is \emph{constant across all entries}.
We can rewrite \cref{eq:white_noise} to see that the data is Gaussian distributed:
\begin{displaymath}
  \xi \sim \mathcal{N}(\mu_i, \sigma) \qtext{with} \mu_i = \mi  \qtext{for all} i \in \Omega.
\end{displaymath}
In this case, the link function between the natural parameter $\mu_i$ and the model $\mi$ is simply the identity, i.e., $\ell(\mu) = \mu$.

From standard statistics, the PDF for the normal distribution $\mathcal{N}(\sigma,\mu)$ is
\begin{displaymath}
  \pdf{x}{\mu,\sigma} = \ifrac*{ e^{{-(x-\mu)^2}\,/\,{2\sigma^2}} }{ \sqrt{ 2 \pi \sigma^2 }}.
\end{displaymath}
Following the framework in \cref{eq:general-function}, the elementwise loss function is
\begin{displaymath}
  f(x,m) = \ifrac{(x-m)^2}{2 \sigma^2} + \tfrac{1}{2} \log(2 \pi \sigma^2)
  .
\end{displaymath}
Since $\sigma$ is constant, it has no impact on the optimization, so we remove those terms to arrive at the standard form
\begin{displaymath}
  f(x,m) = (x-m)^2 \qtext{for} x,m \in \Real.
\end{displaymath}
Note that this final form is no longer strictly a likelihood which has implications for, e.g., using Akaike information criterion (AIC) or the Bayesian information criterion (BIC) to choose the number of parameters.
In the matrix case, the maximum likelihood derivation can be found in \cite{Yo41}.

It is not uncommon to add a nonnegativity assumption on $\M$ \cite{PaTa94,Pa97,Pa97b,LeSe99,WeWe01},
which may correspond to some prior knowledge about the means being nonnegative.

\subsection{Bernoulli distribution and connections to logistic regression}
\label{sec:bernoulli-odds-link}

In this subsection, we propose a loss function for binary data.
A binary random variable $x \in \set{0,1}$ is Bernoulli distributed
with parameter $\rho \in [0,1]$ if $\rho$ is the probability of
a 1 and, consequently, $(1-\rho)$ is the probability of a zero.
We denote this by $x \sim \Bern(\rho)$.
Clearly, the PMF is given by
\begin{equation}\label{eq:bernoulli-pdf}
  \pdf{x}{\rho} = \rho^x (1-\rho)^{(1-x)} \quad x \in \set{0,1}.
\end{equation}
A reasonable model for a binary data tensor $\X$ is
\begin{equation}\label{eq:bernoulli}
  \xi \sim \Bern(\rho_i) \qtext{where} \ell(\rho_i) = \mi.
\end{equation}
If we choose $\ell$ to be the identity link, then we need to constrain
$\mi \in [0,1]$ which is a complex nonlinear constraint, i.e.,
\begin{equation}\label{eq:nonlinear_constraint}
  0 \leq \sum_{j=1}^r a_1(i_1,j) a_2(i_2,j) \cdots a_d(i_d,j)  \leq 1
  \qtext{for all} i \in \I.
\end{equation}
Instead, we can use a different link function.

One option for the link function is to
work with the \emph{odds} ratio, i.e.,
\begin{equation} \label{eq:odds-link}
  \ell(\rho) = \ifrac{\rho}{1-\rho}.
\end{equation}
It is arguably even easier to think in terms of odds ratios than the probability, so this is a natural transformation.
For any $\rho \in [0,1)$, we have $\ell(\rho) \geq 0$. Hence, using
\cref{eq:odds-link} as the link function means that we need only
constrain $\mi \geq 0$. This constraint can be enforced by requiring the
factor matrices to be nonnegative, which is a bound constraint and much
easier to handle than the nonlinear constraint \cref{eq:nonlinear_constraint}.
With some algebra, it is easy to show that we can write the log of \cref{eq:bernoulli-pdf} as
\begin{displaymath}
  -\log \bigl( \pdf{x}{\rho} \bigr) = \logfrac{1}{1-\rho} - x \logfrac{\rho}{1-\rho}.
\end{displaymath}
Plugging this and the link function \cref{eq:odds-link} into our general framework in \cref{eq:general-function}
yields the elementwise loss function
\begin{displaymath}
  f(x,m) = \log(1+m) - x \log m \qtext{for} x \in \set{0,1}, m \geq 0.
\end{displaymath}
For a given odds $m \geq 0$, the associated probability is $\rho = m / (1+m)$.
Note that $f(1,0) = -\infty$ because this represents a statistically impossible situation. In practice, we
replace $\log m$ with $\log (m+\epsilon)$ for some small $\epsilon>0$ to prevent numerical issues.

Another common option for the link function is to work with the log-odds, i.e.,
\begin{equation}\label{eq:log-odds-link}
  \ell(\rho) = \logfrac{\rho}{1-\rho}.
\end{equation}
It is so common that it has a special name: \emph{logit}.
The loss function then becomes
\begin{displaymath}
  f(x,m) = \log(1+e^m) - xm \qtext{for} x\in \set{0,1}, m \in \Real,
\end{displaymath}
and the associated probability is $\rho = e^m/(1+e^m)$.
In this case, $m$ is completely unconstrained and can be any real value.
This is the transformation commonly used in logistic regression.
A form of logistic tensor decomposition for a different type of decomposition called DEDICOM was proposed by Nickel and Tresp \cite{NiTr13}.

We contrast the odds and logit link functions in terms of the interpretation of
the components.
An advantage of odds with nonnegative factors
is that each component can only \emph{increase} the
probability of a 1.
The disadvantage is that it requires a nonnegativity constraint.
The logit link is common in statistics
and has the advantage that it does not require any constraints.
A potential disadvantage is that it may be harder to interpret components
since they can counteract one another.
Moreover, depending on the signs of its factors,
an individual component can simultaneously
increase the probability of a 1 for some entries
while reducing it for others.
As such, interpretations may be nuanced.

\subsection{Gamma distribution for positive continuous data}
\label{sec:gamma-dist}
There are several distributions for handling nonnegative continuous data. As mentioned previously, one option is to assume a Gaussian distribution but impose a nonnegativity constraint. Another option is a Rayleigh distribution, discussed in the next subsection.
Yet another is the gamma distribution (for strictly positive data), with PDF
\begin{equation}\label{eq:gamma-pdf}
  \pdf{x}{k,\theta} = \bigl( \ifrac{ x^{k-1} }{ \Gamma(k)\, \theta^{k} } \bigr) \; e^{-x/\theta} \qtext{for} x > 0,
\end{equation}
where $k>0$ and $\theta>0$ are called the shape and scale parameters respectively and $\Gamma(\cdot)$ is the Gamma function.%
\footnote{The Gamma distribution may alternatively by parameterized by $\alpha = k$ and $\beta=1/\theta$.}
We assume $k$ is \emph{constant across all entries and given}, in which case this is a member of the exponential family of distributions.
For example, $k=1$ and $k=2$ are the exponential and chi-squared distributions, respectively.
If we use the link function $\ell(\theta) = \theta/k$ which induces a positivity constraint $m>0$ as a byproduct\footnote{This also means that we set $m$ to be the expected mean value, i.e., $m = \expt{x} = k\theta$.}, and plug this and \cref{eq:gamma-pdf} into \cref{eq:general-function}
and remove all constant terms (i.e., terms involving only $k$), the elementwise loss function is
\begin{equation}\label{eq:gamma-loss}
  f(x,m) = \log(m) + x/m \qtext{for} x > 0, m > 0.
\end{equation}
In
practice, we use the constraint $m \geq 0$ and replace $m$ with $m+\epsilon$ (with small $\epsilon$) in the
loss function \cref{eq:gamma-loss}.

\subsection{Rayleigh distribution for nonnegative continuous data}
\label{sec:rayleigh-mean-link}
As alluded to in the previous subsection
the \emph{Rayleigh} distribution is a distribution for nonnegative data.
The PDF is
\begin{equation}\label{eq:rayleigh-pdf}
  \pdf{x}{\sigma} = \bigl( \ifrac*{x}{\sigma^2} \bigr) \; e^{-x^2 / (2 \sigma^2)} \qtext{for} x \geq 0,
\end{equation}
where $\sigma > 0$ is called the \emph{scale} parameter.
The link $\ell(\sigma) = \sqrt{\pi/2} \; \sigma$ (corresponding to the mean) induces a positivity constraint on $m$.
Plugging this link and \cref{eq:rayleigh-pdf} into \cref{eq:general-function} and removing the constant terms yields the loss function
\begin{equation}\label{eq:rayleigh-loss}
  f(x,m) = 2\log(m) + \tfrac{\pi}{4} \; (x/m)^2 \qtext{for} x \geq 0, m > 0.
\end{equation}
We again replace $m>0$ with $m \geq 0$ and replace $m$ with $m+\epsilon$ (with small $\epsilon$) in the
loss function \cref{eq:rayleigh-loss}.

\subsection{Poisson distribution for count data}
\label{sec:poiss-ident-link}

If the tensor values are counts, i.e., \emph{natural} numbers
($\Natural=\set{0,1,2,\dots}$), then they can be modelled as a Poisson
distribution, a \emph{discrete} probability distribution
commonly used to describe the number of events that occurred
in a specific window in time, e.g., emails per month.  The PMF for a
Poisson distribution with mean $\lambda$ is given by
\begin{equation}\label{eq:poisson-pdf}
  \pdf{x}{\lambda} = \ifrac*{e^{-\lambda} \lambda^x}{x!} \qtext{for} x \in \Natural.
\end{equation}
If we use the identity link function ($\ell(\lambda) = \lambda$) and
\cref{eq:poisson-pdf} in \cref{eq:general-function} and drop
constant terms, we have
\begin{equation}\label{eq:poisson}
  f(x,m) = m - x \log m \qtext{for} x \in \Natural, m \geq 0.
\end{equation}
This loss function has been studied  previously by Welling and Weber \cite{WeWe01} and Chi and Kolda \cite{ChKo12} in the case of tensor decomposition;  Lee and Seung introduced it in the context of matrix factorizations \cite{LeSe99}.
As in the Bernoulli case, we have a statistical impossibility if $x>0$ and $m=0$, so we make the same correction
of adding a small $\epsilon$ inside the log term.

Another option for the link function is the \emph{log link}, i.e., $\ell(\lambda) = \log \lambda$.
In this case, the loss function becomes
\begin{equation}\label{eq:poisson-log}
  f(x,m) = e^m - xm \qtext{for} x \in \Natural, m \in \Real.
\end{equation}
The advantage of this approach is that $m$ is unconstrained.

\subsection{Negative binomial for count data}
\label{sec:negative-binomial}

Another option for count data is the negative binomial (NegBinom)
distribution. This distribution models the number of trials required before
we experience $r\in \Natural$ failures, given that the probability of failure is
$\rho \in [0,1]$.
The PMF is given by
\begin{equation}\label{eq:nb-pdf}
  \pdf{x}{r,\rho} = {{x+r-1}\choose{k}} \rho^x (1-\rho)^r
  \qtext{for} x \in \Natural.
\end{equation}
If we use the odds link \cref{eq:odds-link} with the probability of failure $\rho$, then the loss function for a given number of failures $r$ is
\begin{displaymath}
  f(x,m) = (r+x) \log (1+m) - x \log m \qtext{for} x \in \Natural, m > 0.
\end{displaymath}
We could also use a logit link \cref{eq:log-odds-link}.
This is sometimes used as an alternative when Poisson is overdispersed.

\subsection{Choosing the loss function}
\label{sec:choos-loss-funct}

Our goal is to give users flexibility in the choice of loss function.
In rare cases where the generation of the data is well understood, the loss function may be easily prescribed.
In most real-world scenarios, however, some guesswork is required.
The choice of fit function corresponds to an assumption on how the data is generated (e.g., according to a Bernoulli distribution) and we further assume that the parameters for the data generation form a low-rank tensor. Generally, users would experiment with several different fit functions and several choices for the model rank.

An overview of the statistically-motivated
loss functions that we have discussed is presented in
\cref{tab:loss-functions}.
The choices of Gaussian,
Poisson with log link, Bernoulli  with logit link, and Gamma with given $k$ are part of the \emph{exponential family} of loss
functions, explored by Collins et al.~\cite{CoDaSc02} in the case of
matrix factorization.
We note that some parameters are assumed to be constant (denoted in blue). For the normal and Gamma distributions, the constant terms ($\sigma$ and $k$, respectively) do not even appear in the loss function. The situation is different for the negative binomial, where $r$ does show up in the loss function.
We have modified the positivity constraints ($m > 0$) to instead be nonnegativity constraints ($m \geq 0$) by adding a small $\epsilon=10^{-10}$ in appropriate places inside the loss functions; these changes are indicated in red. This effectively converts the constraint to $m \geq \epsilon$.
The modification is pragmatic since otherwise finite-precision arithmetic yields in $\pm \infty$ gradient and/or function values.
In the sections that follow, nonnegativity of $\M$ is enforced by requiring that the factor matrices ($\Akset$) be nonnegative.

\begin{table}[th]
  \centering\footnotesize
  \caption{Statistically-motivated loss functions. Parameters in blue are assumed to be constant. Numerical adjustments are indicated in red.}
  \label{tab:loss-functions}
  \renewcommand{\arraystretch}{1.5}
  \begin{tabular}{|l|l|l|l|}
    \hline
    Distribution & Link function  & Loss function & Constraints\\ \hline
    $\mathcal{N}(\mu,\textcolor{constant}{\sigma})$ & $m=\mu$ & $(x{-}m)^2$ & $x,m \in \Real$\\ \hline
    Gamma$(\textcolor{constant}{k},\sigma)$ & $m=k \sigma$ & $x / (m \addfix) + \log (m \addfix)$ & $x>0, m \rge 0$\\ \hline
    Rayleigh$(\theta)$ & $m=\sqrt{\pi/2}\,\theta$ & $2\log(m\addfix) + (\pi/4) (x/(m\addfix))^2$ & $x>0, m \rge 0$\\ \hline
    Poisson$(\lambda)$ & $m=\lambda$ & $m-x \log (m \addfix)$ & $x \in \mathbb{N}, m \rge 0$\\ \cline{2-4}
                       & $m=\log \lambda$ & $e^m-x m$ & $x \in \mathbb{N}$, $m \in \Real$ \\ \hline
    Bernoulli$(\rho)$ & $m={\rho}\,/\,{(1{-}\rho)}$ & $\log(m{+}1) {-} x \log (m\addfix)$ & $x \in \set{0,1}, m \rge 0$ \\ \cline{2-4}
                 & $m=\log ( \rho \,/\,(1-\rho)) $ & $\log(1 {+} e^{m}) - x m$ & $x \in \set{0,1}$, $m \in \Real$ \\ \hline
    NegBinom$(\textcolor{constant}{r},\rho)$ & $m={\rho}\,/\,{(1{-}\rho)}$ & $(r{+}x) \log (1{+}m) - x \log (m\addfix)$ & $x \in \Natural$, $m \rge 0$ \\ \hline
  \end{tabular}
\end{table}

In terms of choosing the loss function from this list, the choice may be dictated by the form of the data.
If the data is binary, for instance, then
one of the Bernoulli choices may be preferred.
Count data may indicate a Poisson or NB distribution.
There are several choices for strictly positive data: Gamma, Rayleigh, and even Gaussian with nonnegativity constraints.

The list of possible loss functions and constraints in \cref{tab:loss-functions} is by no means
comprehensive, and many other choices are possible.  For instance, we
might want to use the \emph{Huber loss} \cite{Hu64}, which is quadratic
for small values of $|x-m|$ and linear for larger values. This is a robust loss
function \cite{HaTiFr09}. The Huber loss is
\begin{equation}\label{eq:huber}
  f(x,m; \Delta) =
  \begin{cases}
    (x-m)^2 & \text{if } |x-m| \leq \Delta, \\
    2\Delta|x-m|-  \Delta^2 & \text{otherwise}.
  \end{cases}
\end{equation}
This formulation has continuous first derivatives and so can be used
in the GCP framework.
Another option is to consider $\beta$-divergences, which have been popular in matrix
and tensor factorizations \cite{CiPh09,CiAm10,FeId11}.
We give the formulas with the constant terms (depending only on $x$) omitted:
\begin{displaymath}
  f(x,m;\beta) =
  \begin{cases}
    \tfrac{1}{\beta} m^{\beta} - \tfrac{1}{\beta-1} xm^{\beta-1} & \text{if } \beta \in \Real \setminus \set{0,1}, \\
    m - x \log m & \text{if } \beta = 1, \\
    \frac{x}{m} + \log m & \text{if } \beta = 0.
  \end{cases}
\end{displaymath}
Referring to \cref{tab:loss-functions}, $\beta=1$ is the same as Poisson loss with the identity link,
and $\beta=0$ is the same as the Gamma loss with the linear link.

We show a graphical summary of all the loss functions in \cref{fig:loss_functions}.
The top row is for continuous data, and the bottom row is for discrete data.
The Huber loss can be thought of as a smooth approximation of an L1 loss.
Gamma, Rayleigh, and $\beta$-divergence are similar, excepting the sharpness of the dip near the minimum.

\begin{figure}[tpb]
  \centering
  \includegraphics[width=\textwidth]{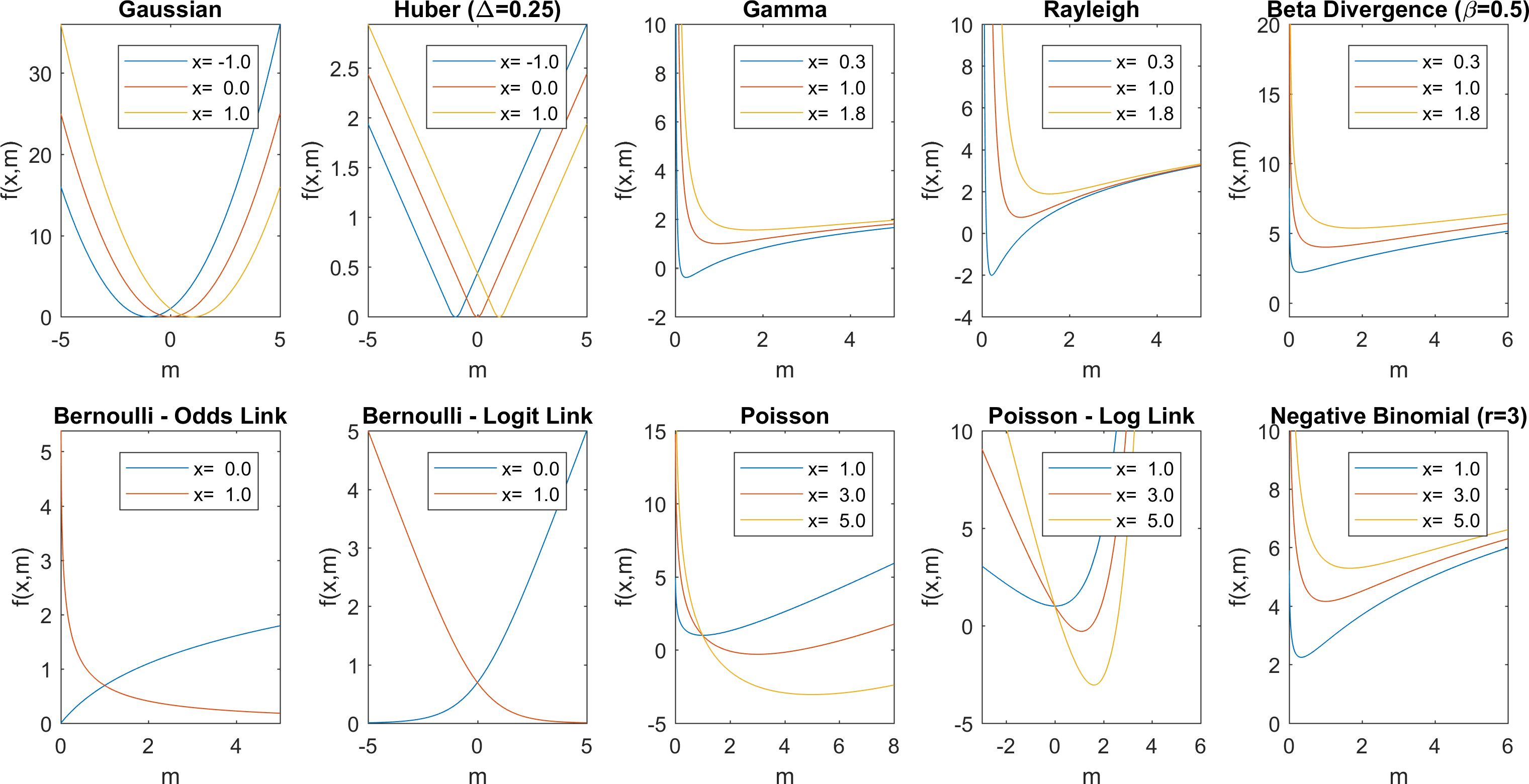}
  \caption{Graphical comparison of different loss functions. Note that some are only defined for binary or integer values of $x$ (bottom row) and that some are only defined for nonnegative values of $x$ and/or $m$.}
  \label{fig:loss_functions}
\end{figure}

\section{GCP decomposition}
\label{sec:fitting-gcp-small}

We now consider how to compute the GCP for a given elementwise loss function.
The majority of this section focuses on dense tensors.
Handling sparse or scarce tensors is discussed in \cref{sec:fitting-gcp-sparse}.

Recall that we have a data tensor $\X$ of size
$n_1 \times n_2 \times \cdots \times n_d$ and that
$\Omega \subseteq \I$ is the set of indices where the values of $\X$
are known.
For a given $r$, the objective for GCP decomposition is to find the
factor matrices $\Ak \in \Real^{n_k \times r}$ for $k=1,\dots,d$ that
solve
\begin{equation}\label{eq:gcp-loss-missing}
  \min \; F(\M; \X, \Omega) \equiv  \frac{1}{|\Omega|} \sum_{i\in\Omega}  f(\xi,\mi) \qtext{subject to} \M = \KT.
\end{equation}
We only sum over the known entries, i.e., $i \in \Omega$;
the same approach to missing data has been used for the CP
decomposition \cite{AcDuKoMo10,AcDuKoMo11}.
We scale by the constant $1/|\Omega|$ so that we are working with the
\emph{mean}. This is simply a convenience that makes it easier to
compare function values for tensors with different sizes or different
amounts of missing data.
This is an optimization problem, and we propose to solve it using
an off-the-shelf optimization method, which has been successful for the standard CP
decomposition \cite{AcDuKo11,PhTiCi13a} and is amenable to missing data \cite{AcDuKoMo10,AcDuKoMo11}.
In contrast to an alternating approach, we do not have to solve a
series of optimization problems.
The main advantage of the alternating least squares in the solution of the
standard CP decomposition is that the subproblems have closed-form solutions \cite{KoBa09};
in contrast, the GCP subproblems do not have closed-form solutions so we do not use an alternating method.

We focus on first-order methods, so we need to calculate the gradient of $F$ with respect to the factor matrices.
This turns out to have an elegant formulation as shown in \cref{sec:gcp-gradient}.
The GCP formulation \cref{eq:gcp-loss-missing} can be augmented in various ways. We might add constraints on the factor matrices such as nonnegativity.
Another option is to add L2-regularization on the factor matrices
to handle the scale ambiguity \cite{AcDuKo11}, and we explain how to do this in \cref{sec:regularization}.
We might alternatively want to use L1-regularization on the factor matrices to encourage sparsity.
The special structure for sparse and scarce tensors is discussed in \cref{sec:fitting-gcp-sparse}.

\subsection{GCP gradient}
\label{sec:gcp-gradient}
We need the gradient of $F$ in \cref{eq:gcp-loss-missing} with respect to the factor matrices,
and this is our main result in \cref{thm:grad}.
The importance of this result is that it shows that the gradient can be calculated via a standard tensor operation called the matricized tensor times Khatri-Rao product (MTTKRP), allowing us to take advantage of existing optimized implementations for this key tensor operation.
Before we get to that, we establish
some useful results in the matrix case. These will be applied to mode-$k$ unfoldings of $\M$ in the proof of \cref{thm:grad}.
The next result is standard in matrix calculus and left as an exercise for the reader.

\begin{lemma}\label{lem:ma}
  Let $\Mx{M} = \A\B'$ where $\A$ is a matrix of
  size $n \times r$ and $\B$ is a matrix of size $p \times r$.
  Then
  \begin{displaymath}
    \FPD{m_{i\ell}}{a_{i'j}} =
    \begin{cases}
      b_{\ell j}
      & \text{if } i = i', \\
      0 & \text{if } i \neq i'
    \end{cases}
    \qtext{for all} i,i'\in\nnset{n}, j\in\nnset{r}, \ell\in\nnset{p}.
  \end{displaymath}
\end{lemma}

Next, we consider the problem of generalized \emph{matrix} factorization in
\cref{lem:grad-matrix}, which is our linchpin result.
This keeps the index notation simple but captures exactly what we need for the main result in \cref{thm:grad}.
In \cref{lem:grad-matrix}, the matrix $\Mx{W}$ is an arbitrary matrix of weights for the terms in the summation, and
the matrix $\Mx{Y}$ (which depends on $\Mx{W}$) is a matrix of derivatives of the elementwise loss function with respect to the model.

\begin{lemma}\label{lem:grad-matrix}
  Let $\Mx{X}, \Mx{W}, \A, \B$  be matrices of size $n \times p$, $n \times p$, $n \times r$, and $p \times r$, respectively.
  Let $f: \Real \times \Real \rightarrow \Real$ be a function that is continuously differentiable w.r.t.\@ its second argument.
  Define the real-valued function $\tilde F$ as
  \begin{equation}\label{eq:gcp-matrix}
    \tilde F(\Mx{M};\Mx{X},\Mx{W}) = \sum_{i=1}^n \sum_{\ell=1}^p w_{i\ell} \, f(x_{i\ell}, m_{i\ell}) \qtext{subject to} \Mx{M} = \A\B'.
  \end{equation}
  Then the first partial derivative of $\tilde F$ w.r.t.\@ $\A$ is
  \begin{displaymath}
    \FPD{\tilde F}{\A} = \Mx{Y} \B \quad \in \Real^{n \times r}
  \end{displaymath}
  where we define the $n \times p$ matrix $\Mx{Y}$ as
  \begin{equation}\label{eq:Y-mat}
    y_{i\ell} = w_{i\ell} \; \FPD{f}{m_{i\ell}}(x_{i\ell},m_{i\ell})
    \qtext{for all} i \in \nnset{n}, \ell \in \nnset{p}.
  \end{equation}
\end{lemma}
\begin{proof}
  Consider the derivative of $\tilde F$ with respect to matrix element $a_{i j}$.
  We have
  \begin{align*}
    \FPD{\tilde F}{a_{i j}}
    &= \sum_{i'=1}^n \sum_{\ell=1}^p w_{i'\ell} \, \FPD{f}{a_{i j}}(x_{i'\ell}, m_{i'\ell}) && \text{by definition of $F$}\\
    &= \sum_{i'=1}^n \sum_{\ell=1}^p w_{i'\ell} \, \FPD{f}{m_{i'\ell}}(x_{i'\ell}, m_{i'\ell}) \FPD{m_{i'\ell}}{a_{i j}} && \text{by chain rule},\\
    &= \sum_{\ell=1}^p y_{i\ell} b_{\ell j} && \text{by \cref{lem:ma} and \cref{eq:Y-mat}}.
  \end{align*}
  Rewriting this in matrix notation produces the desired result.
\end{proof}

Now we can consider the tensor of the GCP problem
\cref{eq:gcp-loss-missing} in \cref{thm:grad}.  For simplicity, we
replace $\Omega$ with an indicator tensor $\W$ such that
$\wi = \dOmega$ and rewrite $F$ using $\W$.
Although this result specifies a specific $\W$, it could
be extended to incorporate general weights
such as the relative importance of each entry; see \cref{sec:concl-future-work}
for further discussion on this topic.

\begin{theorem}[GCP Gradients] \label{thm:grad} %
  Let $\X$ be a tensor of size $n_1 \times n_2 \times \cdots \times n_d$ and $\Omega$ be the indices of known elements of $\X$.
  Let $f: \Real \times \Real \rightarrow \Real$ be a function that is continuously differentiable w.r.t.\@ its second argument.
  Define $\W$ to be an indicator tensor such that $\wi = \dOmega / |\Omega|$.
  Then we can rewrite the GCP problem \cref{eq:gcp-loss-missing} as
  \begin{equation}\label{eq:gcp-loss-w}
    \min \; F(\M; \X, \W) \equiv  \sum_{i\in\I} \wi \,  f(\xi,\mi) \qtext{subject to} \M = \KT.
  \end{equation}
  Here $\Ak$ is a matrix of size $n_k \times r$ for $k \in \nnset{d}$.
  For each mode $k$, the first partial derivative of $F$ w.r.t.\@ $\Ak$ is given by
  \begin{equation}
    \label{eq:GCP-grad}
    \FPD{F}{\Ak} = \Yk \Zk
  \end{equation}
  where $\Zk$ is defined in \cref{eq:Zk} and $\Yk$ is the mode-$k$ unfolding of a tensor $\Y$ defined by
  \begin{equation}
    \label{eq:Y}
    \yi = \wi \, \FPD{f}{\mi}(\xi,\mi)
    \qtext{for all} i \in \I.
  \end{equation}
\end{theorem}

\begin{proof}
  For a given $k$,
  recall that $\Mk = \Ak \Zk'$. Hence, we can write $F$ in \cref{eq:gcp-loss-w} as
  \begin{displaymath}
    F(\M; \X, \W) = \tilde F(\Ak"\Zk'; \Xk, \Wk),
  \end{displaymath}
  where $\tilde F$ is from \cref{eq:gcp-matrix}.
  The result follows from \cref{lem:grad-matrix} with the substitutions
  used in the following table:
  \medskip
  \begin{center}
  \renewcommand{\arraystretch}{1.2}
  \begin{tabular}{|r*{8}{>{$}l<{$}}|}
    \hline
    Matrix Case & \Mx{X} & \Mx{W} & \A & \B & \Mx{Y} & n   & p & r\\
    Tensor Case & \Xk & \Wk & \Ak & \Zk & \Yk & n_k & n^d/n_k & r \\
    \hline
  \end{tabular}
  \end{center}
  \medskip
  We note that the definition of $\Y$ is consistent across all $k$.
\end{proof}

\Cref{thm:grad} generalizes several previous results: the gradient for CP
\cite{ShHa05,AcDuKo11}, the gradient for CP in the case of missing
data \cite{AcDuKoMo11}, and the gradient for Poisson tensor
factorization \cite{ChKo12}.

Consider the gradient in \cref{eq:GCP-grad}.
The $\Zk$ has no dependence on $\X$, $\Omega$, or the loss function; it depends only on the structure of the model.
Conversely, $\Y$ has no dependence on the structure of the model.
The \emph{elementwise derivative tensor} $\Y$ is the same size as $\X$ and is zero wherever $\X$ is missing data.
The structure of $\Omega$ determines the structural sparsity of $\Y$, and this will be important in \cref{sec:fitting-gcp-sparse}.
The form of the derivative is a matricized tensor times Khatri-Rao product (MTTKRP) with the tensor $\Y$ and the Khatri-Rao product $\Zk$.
The MTTKRP is the dominant kernel in the standard CP computation in terms of computation time
and has optimized high-performance implementations \cite{BaKo07,SmKa15,LiMaYaVu16,HaBaJiTo17}.
In the dense case, the MTTKRP costs $\mathcal{O}(rn^d)$.

\Cref{alg:gcp-fg} computes the GCP loss function and gradient.
On \cref{line:gcp-fg-full} we compute elementwise values at known data locations.
If all or most elements are known, we can compute the full model using~\cref{eq:Zk} at a cost of $rn^d$.
However, if only a few elements are known, it may be more efficient to compute model values only at the locations in $\Omega$ using \cref{eq:M-low-rank} at a cost of $2r|\Omega|$.
We compute the elementwise derivative tensor $\Y$ in \cref{line:gcp-fg-y}; here the quantity $\dOmega$ is 1 if $i\in \Omega$ and 0 otherwise.
The cost of \cref{line:gcp-fg-F,line:gcp-fg-y} is $O(|\Omega|)$.
\Cref{line:gcp-fg-mttkrp-start,line:gcp-fg-mttkrp,line:gcp-fg-mttkrp-end} compute the gradient with respect to each factor matrix, and the
cost is $O(drn^d)$.
Communication lower bounds as well as a parallel implementation for MTTKRP for dense tensors are covered in \cite{BaKnRo17}.
Since this is a \textit{sequence} of MTTKRP operations, we can also consider reusing intermediate computations as has been done \cite{PhTiCi13} and reduces the $d$ part of the expense.
Hence, the cost is dominated by the MTTKRP, just as for the standard CP-ALS.
We revisit this method in the case of sparse or large-scale tensors in \cref{sec:fitting-gcp-sparse}.

\begin{algorithm}
\caption{GCP loss function and gradient}
\label{alg:gcp-fg}
   \begin{algorithmic}[1]
     \Function{gcp\_fg}{$\X$,$\Omega$,$\Akset$}
     \State\label{line:gcp-fg-full} $\mi \gets \textsc{entry}(\Akset,i)$ for all $i \in \Omega$ \Comment{Model entries}
     \State\label{line:gcp-fg-F} $F \gets \frac{1}{|\Omega|}\sum_{i \in \Omega} f(\xi,\mi)$ \Comment{Loss function}
     \State\label{line:gcp-fg-y} $\yi \gets (\dOmega/|\Omega|) \, \FPD{f}{\mi}(\xi,\mi)$  for all $i \in \I$
     \Comment{Elementwise derivative tensor}
     \For{\label{line:gcp-fg-mttkrp-start}$k = 1, \dots, d$}  \Comment{Full sequence of MTTKRPs}
     \State\label{line:gcp-fg-mttkrp} $\Gk \gets \textsc{mttkrp}(\Y, \KT, k)$ \Comment{Gradients w.r.t.~$\Ak$}
     \EndFor\label{line:gcp-fg-mttkrp-end}
     \State \Return $F$ and $\Gkset$
     \EndFunction
   \end{algorithmic}
\end{algorithm}

\subsection{Regularization}
\label{sec:regularization}

It is straightforward to add regularization to the GCP formulation.
This may especially be merited when there is a large proportion of missing data, in which case some of the factor elements may not be constrained due to lack of data. As an example, consider simple L2 regularization.
We modify the GCP problem in \cref{eq:gcp-loss-missing} to be
\begin{multline}\label{eq:gcp-loss-missing-regularization}
  \min F(\M; \X, \Omega, \set{\eta_k}) \equiv
  \frac{1}{|\Omega|} \sum_{i \in \Omega} f(\xi,\mi) + \sum_{k=1}^d \frac{\eta_k}{2} \| \Ak \|_2^2 \\
  \qtext{subject to} \M = \KT.
\end{multline}
In this case, the gradients are given by
\begin{equation}\label{eq:gcp-grad-regularization}
  \FPD{F}{\Ak} = \Yk \Zk + \eta_k \Ak,
\end{equation}
where $\Yk$ and $\Zk$ are the same as in \cref{eq:Y}.
The difficulty is in picking the regularization parameters, $\set{\eta_k}$.
These can all be equal or different, and can be selected by cross-validation using prediction of held out elements.

\subsection{GCP Decomposition for Sparse or Scarce Tensors}
\label{sec:fitting-gcp-sparse}

We say a tensor is \emph{sparse} if the vast majority of its entries
are zero. In contrast, we say a tensor is \emph{scarce} if the vast
majority of its entries are missing.  In either case, we can store
such a tensor efficiently by keeping only its nonzero/known values and
the corresponding indices. If $s$ is the number of nonzero/known values,
the required storage is $s(d+1)$ rather than $n^d$ for the dense tensor
where every zero or unknown value is stored explicitly.

The fact that $\X$ is sparse does not imply that the $\Y$ tensor needed to compute the
gradient (see \cref{thm:grad}) is sparse.
This is because $\FPD{f}{\mi}(0,\mi) \neq 0$ for general values of $\mi$.
There are two cases where the gradient has a structure that allows
us to avoid explicitly calculating $\Y$:
\begin{itemize}
\item Standard Gaussian formulation; see \cref{sec:spec-struct-stand} for details.
\item Poisson formulation with the identity link; see \cite{ChKo12} for details.
\end{itemize}
Otherwise, we have to calculate the dense $\Y$ explicitly in order to compute the gradients.
For many large-scale tensors, this is infeasible.
The fact that $\X$ is scarce, however, does imply that  the
tensor $\Y$ is sparse.
This is because all missing elements in $\X$ correspond to zeros in $\Y$.

Let us take a moment to contrast the implication of sparse versus scarce.
Recall that a sparse tensor
is one where the vast majority of elements are zero, whereas a scarce
tensor is one where the vast majority of elements are missing. The
elementwise gradient tensor $\Y$ for a sparse tensor is structurally dense, but it
is sparse for a scarce tensor. To put it another way, if $\X$ is
sparse, then the MTTKRP calculation in \cref{line:gcp-fg-mttkrp} of
\cref{alg:gcp-fg} has a dense $\Y$; but if $\X$ is scarce, then the
MTTKRP calculation uses a sparse $\Y$. Further discussion of sparse
versus scarce in the matrix case can be found in a blog post by Kolda \cite{Ko17}.
We summarize the situation in \cref{fig:sparse_vs_scarse}.

\begin{figure}[ht]
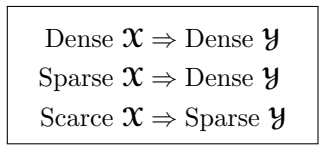

  \centering
  \fbox{%
    \begin{minipage}{0.3\linewidth}%
      \vspace{-1em}%
    \begin{align*}
    \text{Dense } \X & \Rightarrow \text{Dense } \Y \\
    \text{Sparse } \X & \Rightarrow \text{Dense } \Y \\
    \text{Scarce } \X & \Rightarrow \text{Sparse } \Y
    \end{align*}
    \end{minipage}}
  \caption{Contrasting sparsity and scarcity in GCP.}
  \label{fig:sparse_vs_scarse}
\end{figure}

The idea that scarcity yields sparsity in the gradient calculation suggests several possible
approaches for handling large-scale tensors.  One possibility is to
simply leave out some of the data, i.e., impose scarcity.  Consider
that we have a vastly overdetermined problem because we have $n^d$
observations but only need to determine $rd\bar n$ parameters.
Special care needs to be taken if the tensor is sparse, since leaving
out the vast majority of the nonzero entries would clearly degrade the
solution.  Another option is to consider stochastic gradient descent,
where the batch at each iteration can be considered as a scarce
tensor, leading again to a sparse $\Y$ in the gradient calculation.
These are topics
that we will investigate in detail in future work.

\section{Experimental results}
\label{sec:experimental-results}
The goal of GCP is to give data analysts the flexibility to try out
different loss functions. In this section, we show examples that
illustrate the
differences in the tensor factorization from using different loss
functions.
We do not claim that any particular loss function is better than any
other; instead, we want to highlight the ability to easily use different loss
functions.
Along the way, we also show the general utility of tensor decomposition, which includes:
\begin{itemize}
\item \textbf{Data decomposition into explanatory factors}:
  We can directly visualize the resulting components and oftentimes use this for interpretation.
  This is analogous
  to matrix decompositions such as
  principal component analysis, independent component analysis,
  nonnegative matrix factorization, etc.
\item \textbf{Compressed object representations}: Object $i_k$ in mode $k$
  corresponds to row $i_k$ in factor matrix $\Ak$, which is a length-$r$ vector.
  This can be used as input to regression, clustering, visualization, machine learning, etc.
\end{itemize}
We focus primarily on these types of activities. However, we could
also consider filling in missing data, data compression, etc.

All experiments are conducted in MATLAB.
The method is implemented as \texttt{gcp\_opt} in the Tensor Toolbox
for MATLAB \cite{TensorToolbox,BaKo06}.
For the optimization, we use limited-memory BFGS
with bound constraints (L-BFGS-B) \cite{ByLuNoZh95}%
\footnote{We specifically use the MATLAB-compatible translation by Stephen Becker,
  available at \url{https://github.com/stephenbeckr/L-BFGS-B-C}}.
First-order optimization methods such as L-BFGS-B typically expect a vector-valued function $f:\Real^n \rightarrow \Real$ and a corresponding vector-valued gradient, but the optimization variables in GCP are matrix-valued; see \cref{sec:gcp-optimization} for discussion of how we practically handle the required reshaping.
For simplicity, we choose a rank that works reasonably well for the purposes
of illustration. Generally, however, the choice of model rank is a complex procedure.
It might be selected based on model consistency across multiple runs,
cross-validation for estimation of hold-out data, or some
prediction task using the factors.
Likewise, we choose an arbitrary ``run'' for the purposes of illustration.
These are nonconvex optimization problems, and so we are not guaranteed that every
run will find the global minimum. In practice, a user would do a few runs and usually
choose the one with the lowest objective value.

\subsection{Social network}
\label{sec:social-net}

We consider the application of GCP to a social network dataset.
Specifically, we use a chat network from students at UC Irvine \cite{SNET,PaOpCa09, OpPa09}.
It contains transmission times and sizes of 59,835 messages sent among 1899 anonymized users over 195 days from April to October 2004.
Because many of the users included in the dataset sent few messages, we select only the 200 most prolific senders in this analysis.
We consider a three-way binary tensor of size $200 \times 200 \times 195$ of the following form:
\begin{displaymath}
  x(i_1, i_2, i_3) =
  \begin{cases}
    1 & \text{if student $i_1$ sent a message to student $i_2$ on day $i_3$,}\\
    0 & \text{otherwise}.
  \end{cases}
\end{displaymath}
It has 9764 nonzeros, so it is only 0.13\% dense though we treat it as dense in this work.
The number of interacting pairs per day is shown in \cref{fig:snet_daily}, and there is clearly more activity earlier in the study.
To give a sense of how many days any given pair of students interact, we consider the histogram in \cref{fig:snet_hist}.
The vast majority of students that interacted had only one interaction, i.e., $4 \times 10^4$ of the interactions were for only one day. The maximum number of interaction days was 33, which occurred for only one pair.

\begin{figure}[tpb]
  \centering
  \subfloat[Number of interacting pairs per day. Note the gap around day 70 and the decrease in activity toward the end of the experiment.]{\label{fig:snet_daily}%
    \includegraphics[width=0.48\textwidth]{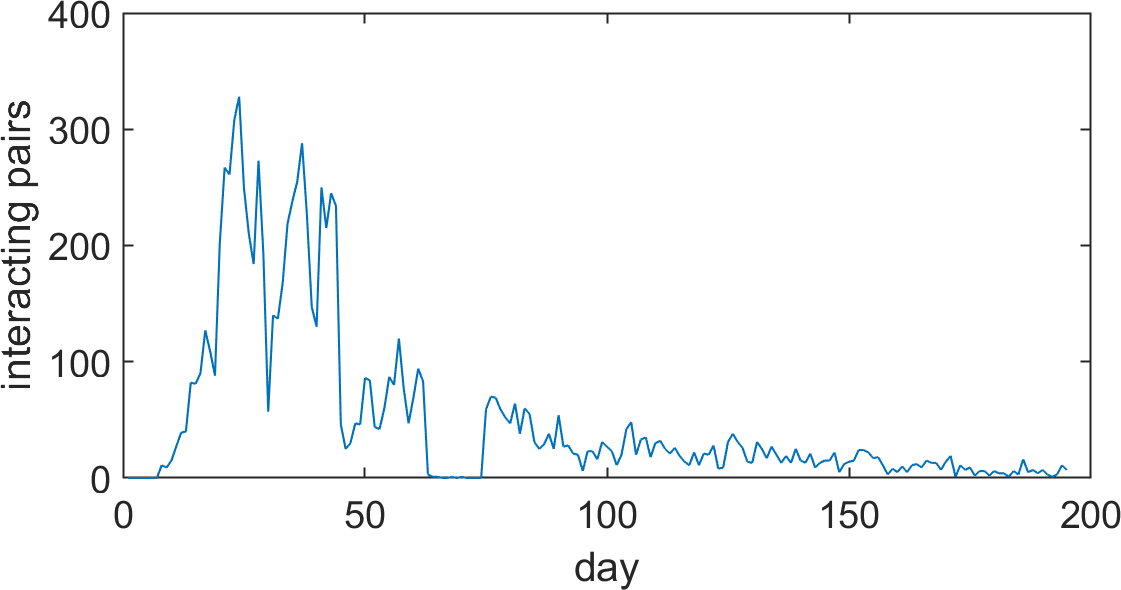}}~
  \subfloat[Histogram of number of interactions per pair where count is in the log scale. Most students only interact once. The greatest number of interaction days is 33.]{\label{fig:snet_hist}%
    \includegraphics[width=0.48\textwidth]{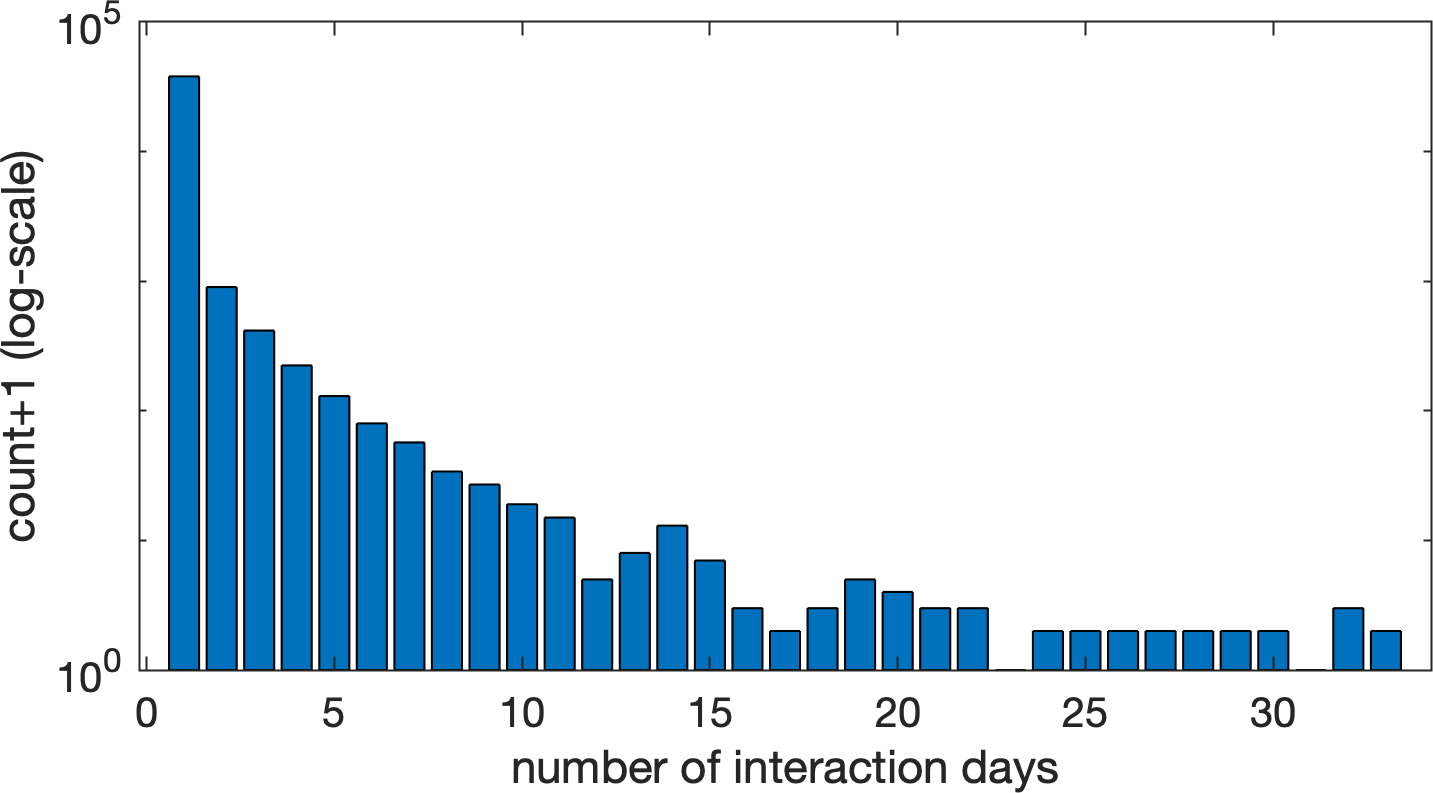}}
  \caption{Statistics for a social network tensor where $x(i_1,i_2,i_3)=1$ if student $i_1$ sends a message to student $i_2$ on day $i_3$.}
  \label{fig:snet_background}
\end{figure}

\subsubsection{Explanatory factors for social network}
\label{sec:explanatory-factors-snet}

We compare the explanatory GCP factors using three different loss functions in \cref{fig:snet}.
 Recall
that each \emph{component} is the outer product of three vectors;
these vectors are what we plot to visualize the model.
In all cases, we use $r=7$ components because it seemed to be adequately descriptive.
To visualize the factorization, components
are shown as ``rows'', numbered on the left, and ordered by magnitude.
We show all three modes as bar plots.
The first two modes correspond to students, as senders and receivers.
They are ordered from greatest to least total activity and normalized to unit length.
The third mode is day, and it is normalized to the magnitude of the component.
Each component groups students that are messaging one another along with the dates of activity.
Each loss function yields a different grouping and so a different interpretation.
The appropriateness
of any particular interpretation depends on the context.

\begin{figure}[tpb]
  \centering
  \subfloat[Gaussian (standard CP). Some factors only pick up one or two students as senders or receivers.]{\label{fig:snet_gaussian}%
  \includegraphics[width=\textwidth]{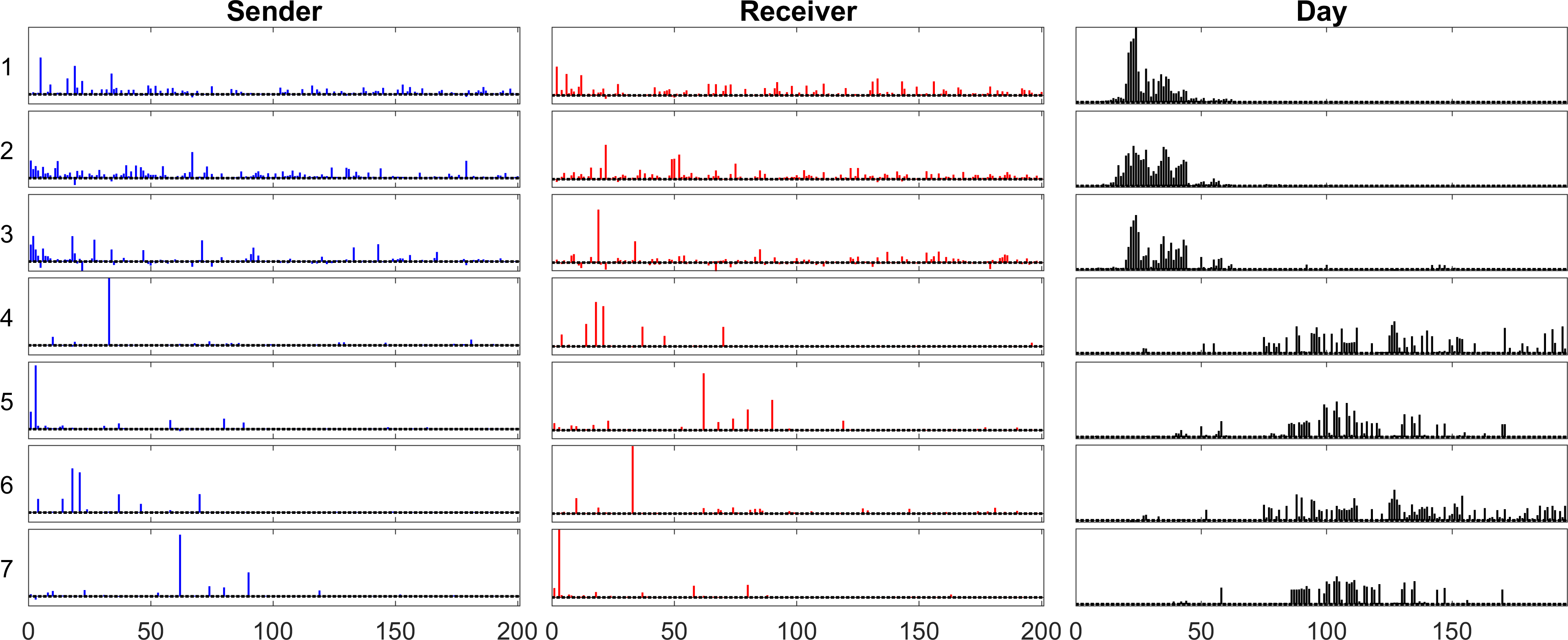}}\\
  \subfloat[Bernoulli-odds (with nonnegativity constraints). Compared with CP-ALS, many students are identified with each component and more emphasis is placed on the heavier traffic days.]{\label{fig:snet_binary}%
  \includegraphics[width=\textwidth]{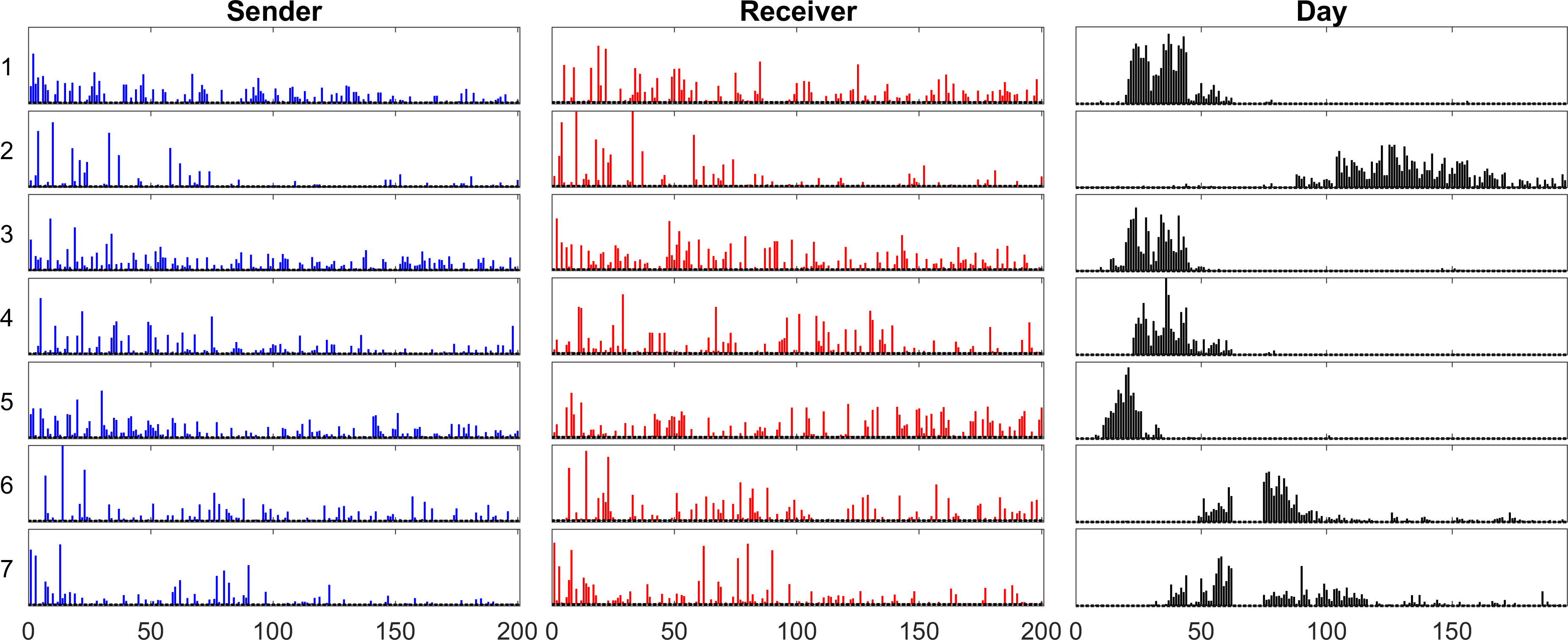}}\\
  \subfloat[Bernoulli-logit. A negative product means the likely result is a zero, i.e., no communication. The first few factors are focused primarily on the zeros.]{\label{fig:snet_binary_logit}%
  \includegraphics[width=\textwidth]{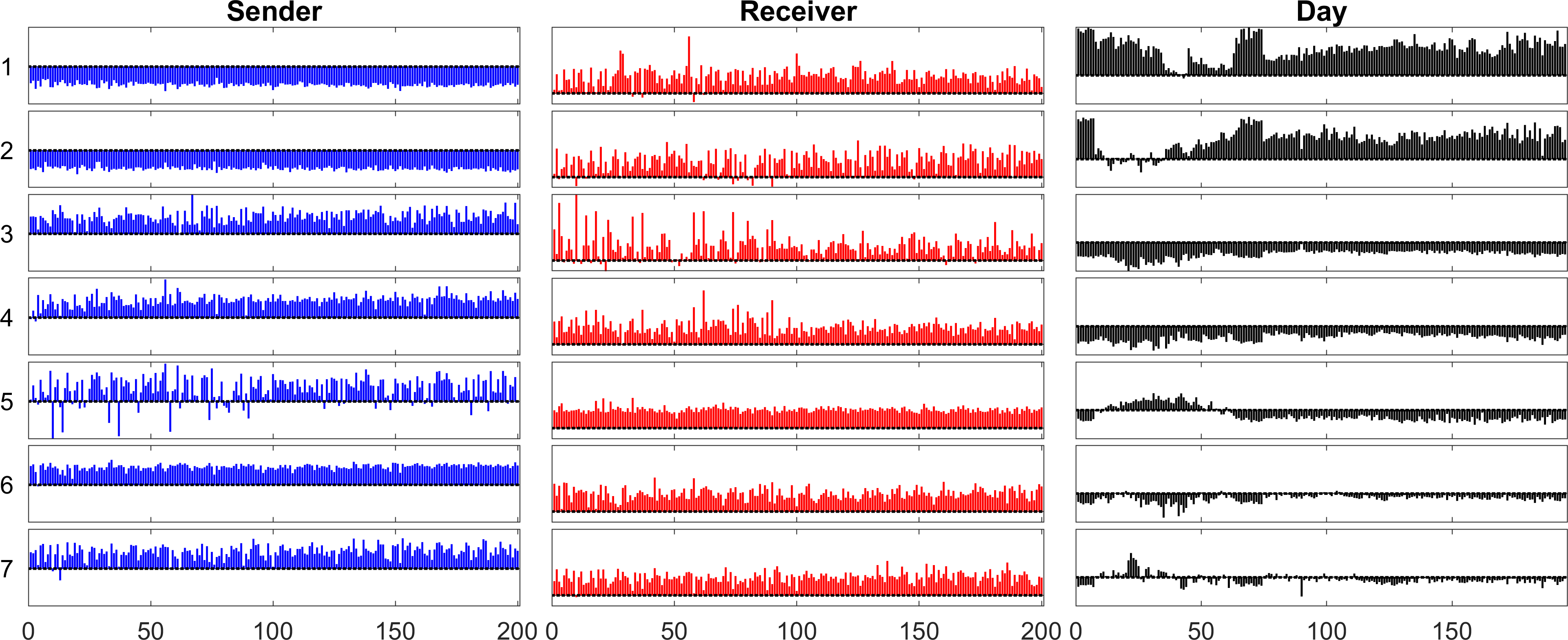}}
  \caption{GCP tensor decomposition of $200 \times 200 \times 195$ binary (0/1) social network tensor using different loss functions and $r=7$.}
  \label{fig:snet}
\end{figure}

For the standard CP in \cref{fig:snet_gaussian},
we did not add a nonnegative constraint on the factors, but
there are only a few small negative entries (see, e.g., the third component).
There is a clear temporal locality in the first three factors.
The remaining four are more diffuse.
A few sender/receiver factors  capture only a few large magnitude entries:
sender factor 4, receiver factor 6, and both sender/receiver factors 7.

For Bernoulli with an odds link in \cref{fig:snet_binary}, the
factor matrices are constrained to be nonnegative.
We see even more defined temporal locality in this version. In particular,
components 6 and 7 do not really have an analogue in the Gaussian version.
The sender and receiver factors are correlated with one another in components 2, 6, and 7, which is something that
we did not really see in the Gaussian case. Such correlations are indicative of a group talking to itself.
The factors in this case seem to do a better job capturing the activity on the most active days per \cref{fig:snet_daily}.

For Bernoulli with a logit link in \cref{fig:snet_binary_logit}, the interpretation is very different.
Recall that negative values correspond to observing zeros. The first component is
roughly inversely correlated with the activity per day, i.e., most entries are zeros and this is what is picked up.
It is only really in components 5 and 7 where there is some push toward positive values, i.e., interactions.

\subsubsection{Prediction for social network}
\label{sec:prediction-snet}
To show the benefit of using a different loss function, we consider the problem of predicting missing values.
We run the same experiment as before but hold out 50 ones and 50 zeros at random when fitting the model.
We then use the model to predict the held out values.
Let $\Omega$ denote the set of known values, so $i \not \in \Omega$ means that the entry was held out.
We measure the accuracy of the prediction using the log-likelihood
under a Bernoulli assumption, i.e., we compute
\begin{displaymath}
  \text{log-likelihood} = \sum_{{x_i=1}\atop{i \not \in \Omega}} \log p_i + \sum_{{x_i=0}\atop{i \not \in \Omega}} \log (1-p_i) ,
\end{displaymath}
where $p_i$ is the \emph{probability} of a one as predicted by the model.
A higher log-likelihood indicates a more accurate prediction.
We convert the predicted values $m_i$, computed from \cref{eq:M-low-rank}, to probabilities $p_i$ (truncated to the range $[10^{-16},1-10^{-16}]$) as follows:
\begin{itemize}
\item \textbf{Gaussian.} Let $p_i = m_i$, truncating to the range (0,1).
\item \textbf{Bernoulli-odds.} Convert from the odds ratio:  $p_i = m_i / (1+m_i)$.
\item \textbf{Bernoulli-logit.} Convert from the log-odds ratio: $p_i = e^{m_i} / (1+e^{m_i})$.
\end{itemize}
We repeat the experiment two hundred times, each time holding out a different set of 100 entries.
The results are shown in \cref{fig:snet-prediction}.
This is a difficult prediction problem since ones are extremely rare;
the differences in prediction performance were negligible for predicting the zeros but predicting the ones was much more difficult.
Both Bernoulli-odds and Bernoulli-logit consistently outperform the standard approach
based on a Gaussian loss function.
We also note that the Gaussian-based predictions were outside of the range $[0,1]$ for 11\% of the predictions, making it tricky to interpret the Gaussian-based predictions.

\begin{figure}
  \centering
  \subfloat[Prediction of 100 missing entries for 200 trials.]{
    \includegraphics[height=2in]{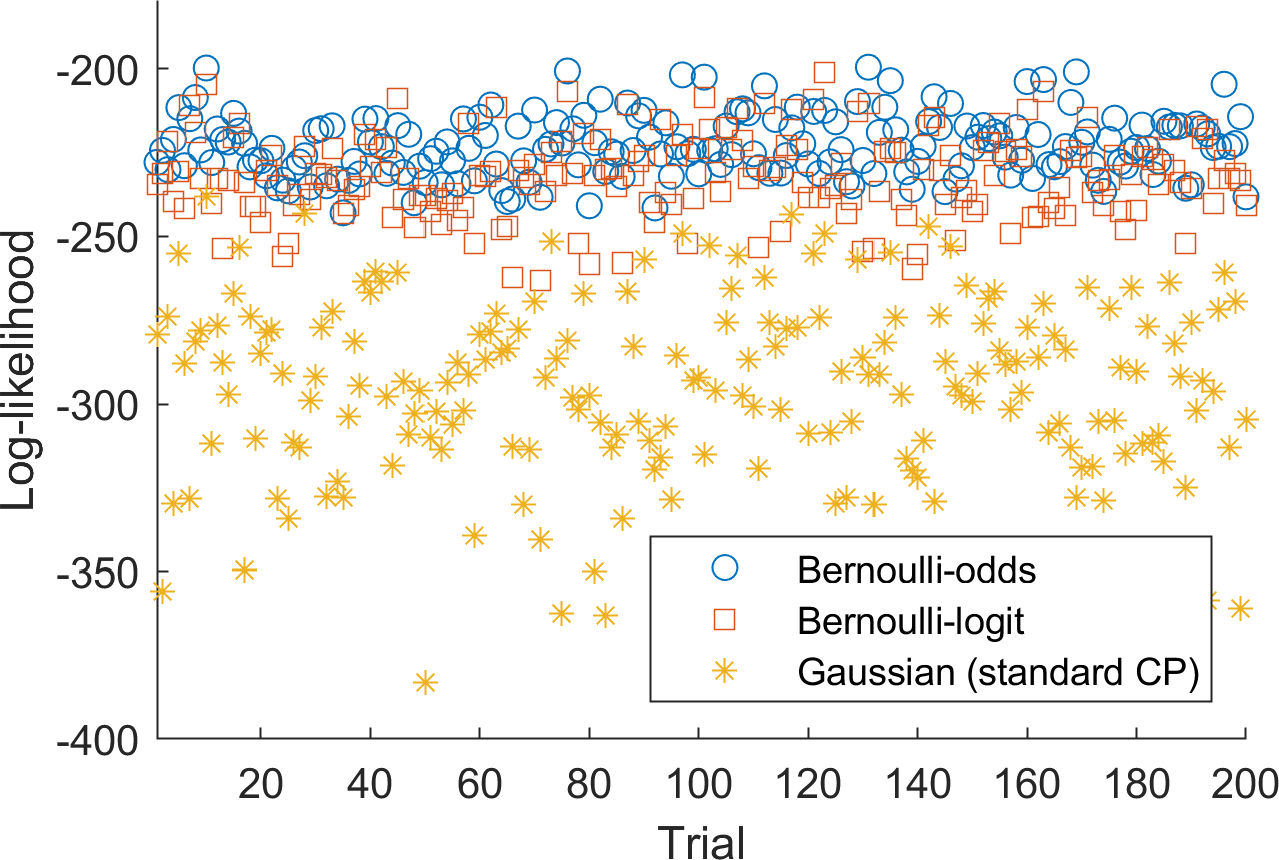}
    }
  \subfloat[Box plot of prediction results.]{
    \includegraphics[height=2in]{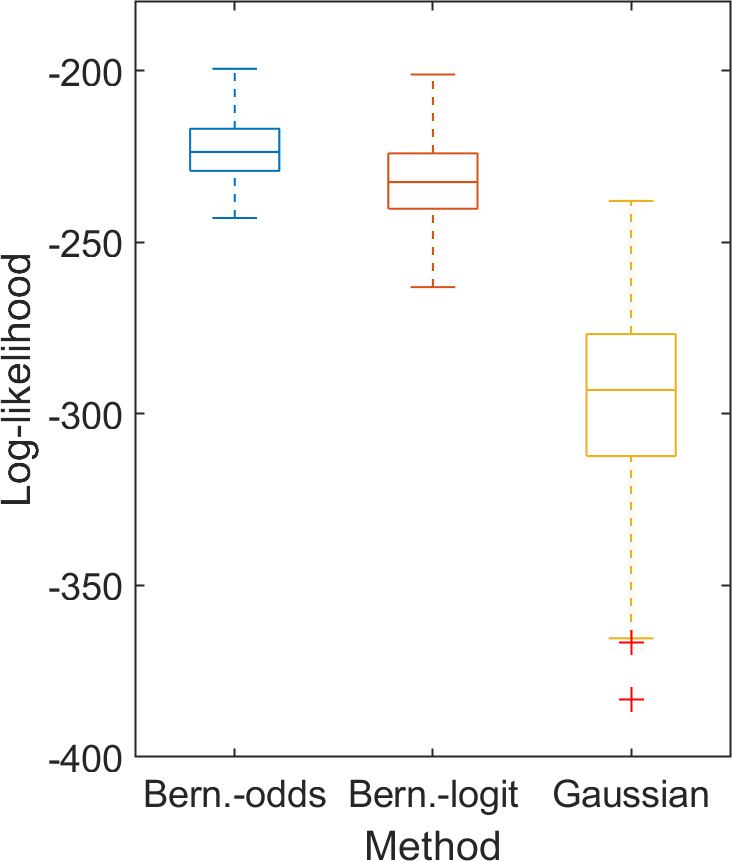}
    }
    \caption{Log-likelihood for GCP with different loss functions. Each trial holds out 50 ones and 50 zeros at random. The GCPs are computed and used to estimate each held-out value. A higher log-likelihood indicates a better prediction.
In the box plot, the box represents 25th--75th percentiles with a horizontal midline at the 50th percentile, i.e., the median. The whiskers extend to the most extreme data points that are not considered outliers, and then outliers are indicated with plus-symbols.
    }
  \label{fig:snet-prediction}
\end{figure}

\subsection{Neural activity of a mouse}
\label{sec:neur-activ-mouse}

In recent work, Williams et~al.~\cite{WiKiWaVy18} consider the
application of CP tensor decomposition to analyze the neural activity of a
mouse completing a series of trials.
They have provided us with a reduced version of their dataset to
illustrate the utility of the GCP framework.
In the dataset we study, the setup is as follows.
A mouse runs a maze over and over again, for a total of 300 trials.
The maze has only one junction, at which point the mouse must turn either right or left.
The mouse is forced to learn which way to turn in order to receive a reward.
For the first 75 trials, the mouse gets a reward if it turns right;
for the next 125 trials, it gets a reward if it turns left; and
for the final 100 trials, it gets a reward if it turns right.
Data was recorded from the prefrontal cortex of a mouse using calcium imaging;
specifically, the activity of 282 neurons was recorded and processed so that all data values lie between 0 and 1.
The neural activity in time for a few sample neurons is shown in \cref{fig:neuron_examples}; we plot each of the 300 different trials and the average value.
From this image, we can see that different neurons have distinctive patterns of activity. Additionally, we see an example of at least one neuron that is clearly active for some trials and not for others (Neuron 117).

\begin{figure}[htb]
  \centering
  \includegraphics[width=\textwidth]{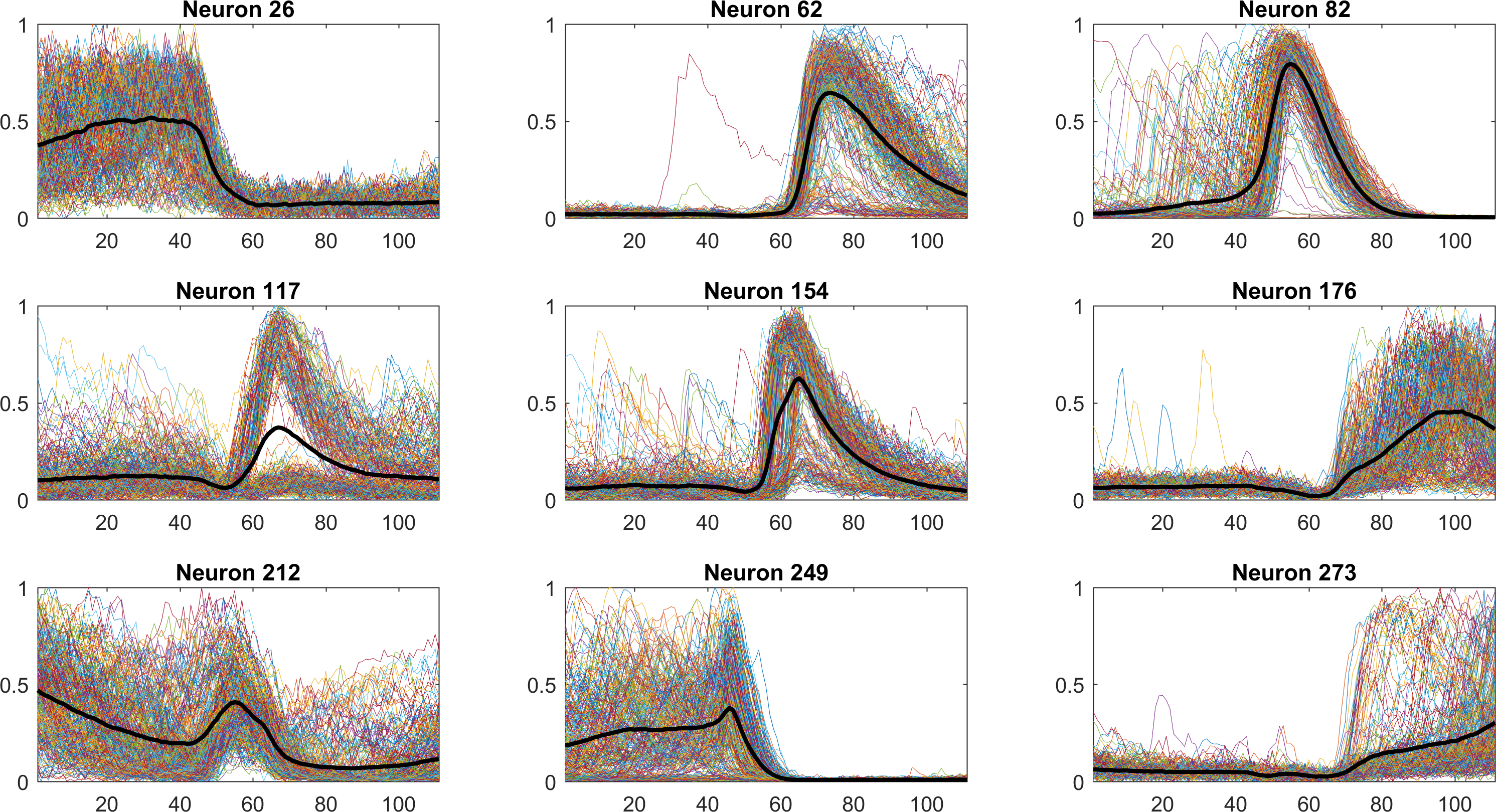}
  \caption{Example neuron activity across all trials. Each thin line (randomly colored) is the time profile for a single trial, and the single dark line is the average over all 300 trials. Different neurons have distinctive temporal patterns. Moreover, some have markedly different activity for different trials, like Neuron 117.}
  \label{fig:neuron_examples}
\end{figure}

This is large and complex multiway data.
We can arrange this data as a three-way nonnegative tensor as follows: 282 (neurons)
$\times$ 110 (time points) $\times$ 300 trials.
Applying GCP tensor decomposition reduces the data into explanatory factors, as we discuss in \cref{sec:explanatory-factors}.
We
show how the factors can be used in a regression task in \cref{sec:regression-task}.

\subsubsection{Explanatory factors for mouse neural activity}
\label{sec:explanatory-factors}

We compare the results of using different loss functions in terms of explanatory factors.
In all cases, we use $r=8$ components.
The first mode corresponds to the neurons and is normalized to the size of the component,
The second and third modes are, respectively, within-trial time and trial, each normalized to length 1.
The neuron factors are plotted as bar graphs,
showing the activation level of each neuron.  We emphasize in red the
bars that correspond to the example neurons from
\cref{fig:neuron_examples}.  The time factors are plotted as lines,
and turn out to be continuous because \emph{that is an inherent
  feature of the data itself}.  We did nothing to enforce continuity
in those factors.  The trial factors are scatter plots, color coded to
indicate which way the mouse turned.  The dot is filled in if the
mouse received a reward. When the rules changed (at trial
75 and 200, indicated by vertical dotted lines), the mouse took several trials to figure out the new
way to turn for the reward.

The result of a
standard CP analysis is shown in \cref{fig:mouse_gaussian}.
Several components are strongly correlated with the trial conditions, indicating the power of the CP analysis.
For instance, component 3 correlates with receiving a reward (filled).
Components 5, 6, and 8 correlate to turning left (orange) and
right (green).
Their time profiles align with when these activities are happening (e.g., end of trial for reward and mid-trial for turn).
The problem with the standard CP model is that interpretation of the
negative values is difficult.
Consider that neuron 212 has a significant score for nearly every component, making it hard to understand its role.
Indeed, several of the example neurons have high magnitude scores for multiple components, and so it might be hard to hypothesize
which neurons correspond to which trial conditions.

\begin{figure}[tp]
  \centering
  \subfloat[Gaussian (standard CP) is difficult to interpret because of negative values]{\label{fig:mouse_gaussian}%
    \includegraphics[width=\textwidth]{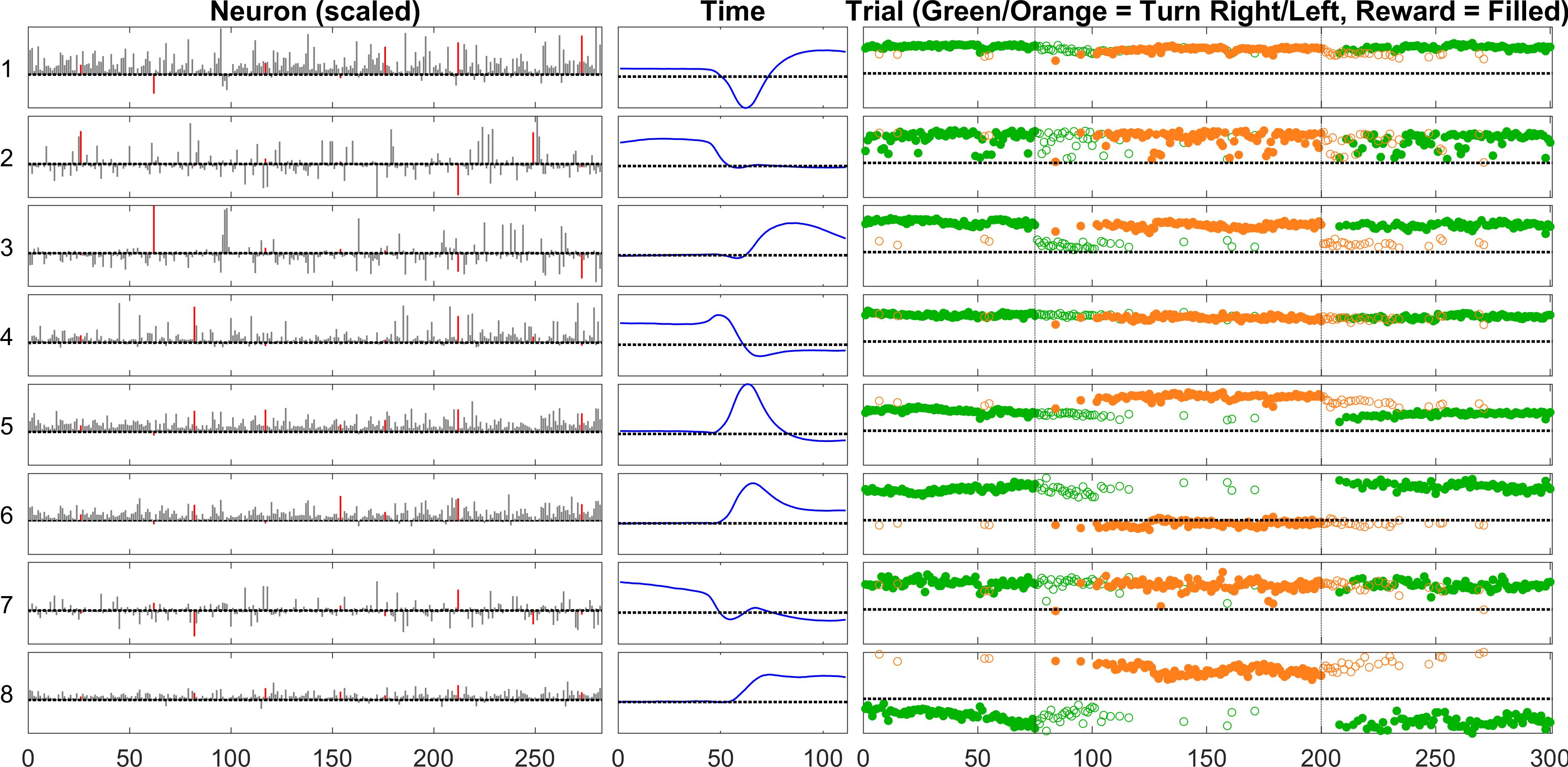}}\\
    \subfloat[Beta-divergence with $\beta{=}0.5$ has no negative factor values and so is easier to interpret]{\label{fig:mouse_beta}%
    \includegraphics[width=\textwidth]{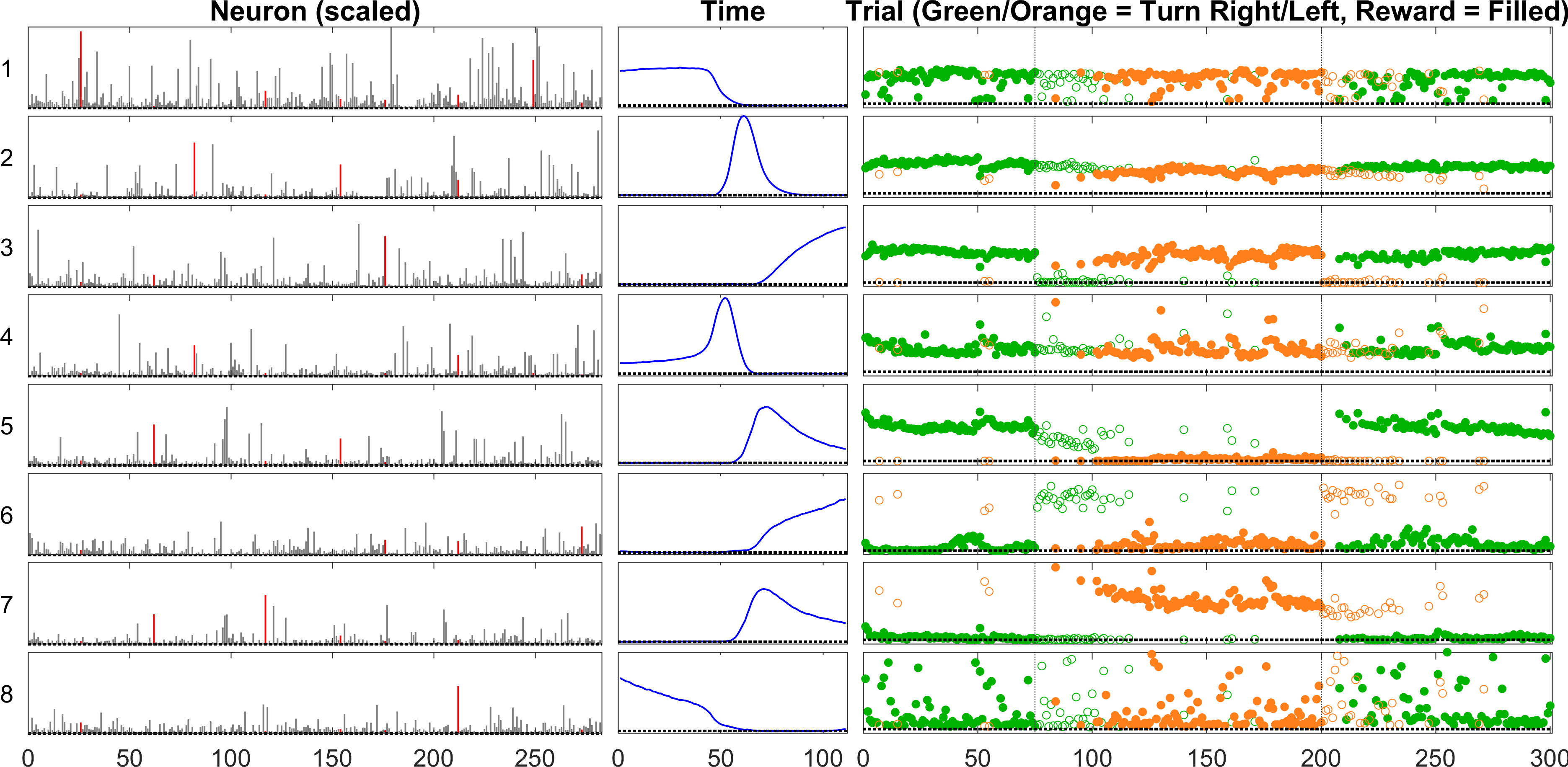}}
  \caption{GCP tensor decomposition of mouse neural activity. \MouseCaption}
\end{figure}

In contrast, consider \cref{fig:mouse_beta} which shows the results of GCP with $\beta$-divergence with $\beta=0.5$.
The factorization is arguably easier to interpret since it has only nonnegative values.
As before, we see that several components clearly correlate with the trial conditions.
Components 3 and 6 correlate with reward conditions.
Components 5 and 7 correlate to the turns.
In this case, the example neurons seem to have clearer identities with the factors.
Neuron 176 is strongest for factor 3 (reward), whereas
neuron 273 is strongest for factor 6 (no reward).
Some of the components do not correspond to the reward or turn, and we do not always know how to interpret them.
They may have to do with external factors that are not recorded in the experimental metadata.
We might also hypothesize interpretations for some  components. For instance, the second component is active mid-trial and may have to do  with detecting the junction in the maze.

For further comparison, we include the results of using Rayleigh,
Gamma, and Huber loss functions in
\cref{fig:mouse_other}.
These capture many of the same trends.

\begin{figure}[tp]
  \centering
  \subfloat[Rayleigh with nonnegativity constraints]{\label{fig:mouse_rayleigh}%
    \includegraphics[width=0.9\textwidth]{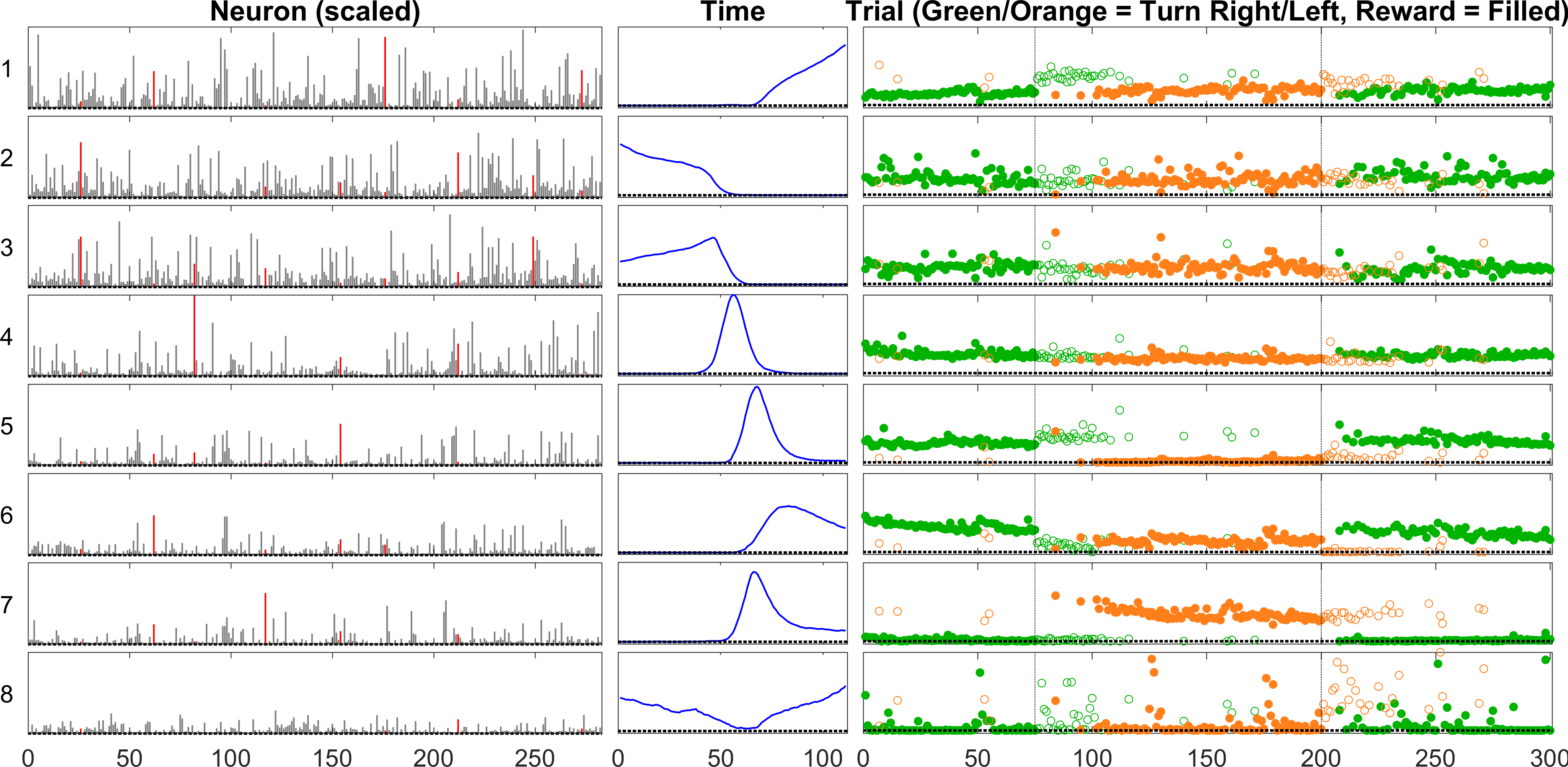}}\\
  \subfloat[Gamma with nonnegativity constraints]{\label{fig:mouse_gamma}%
    \includegraphics[width=0.9\textwidth]{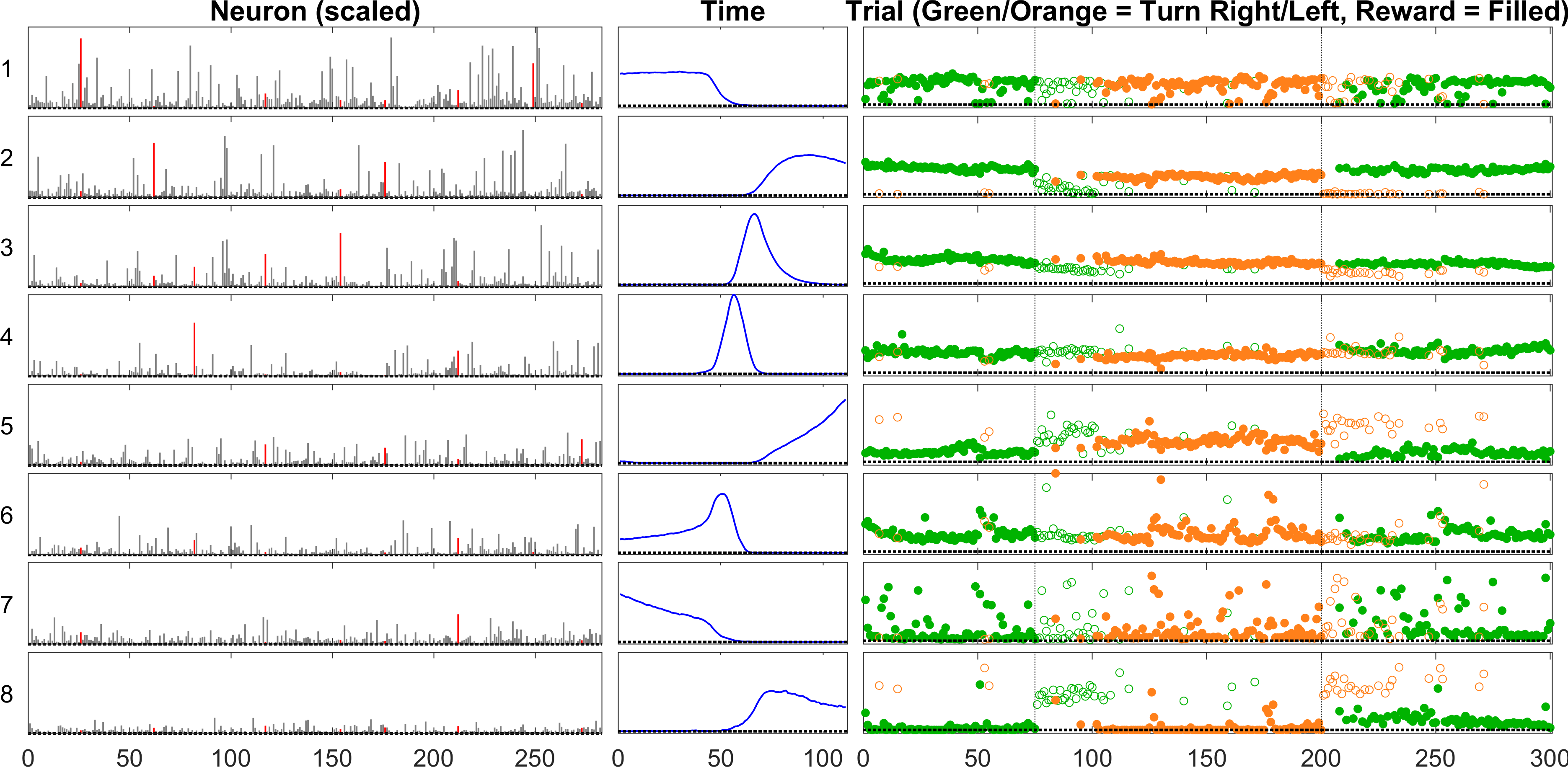}}\\
  \subfloat[Huber with $\Delta{=}0.25$ and nonnegativity constraints]{\label{fig:mouse_huber}%
    \includegraphics[width=0.9\textwidth]{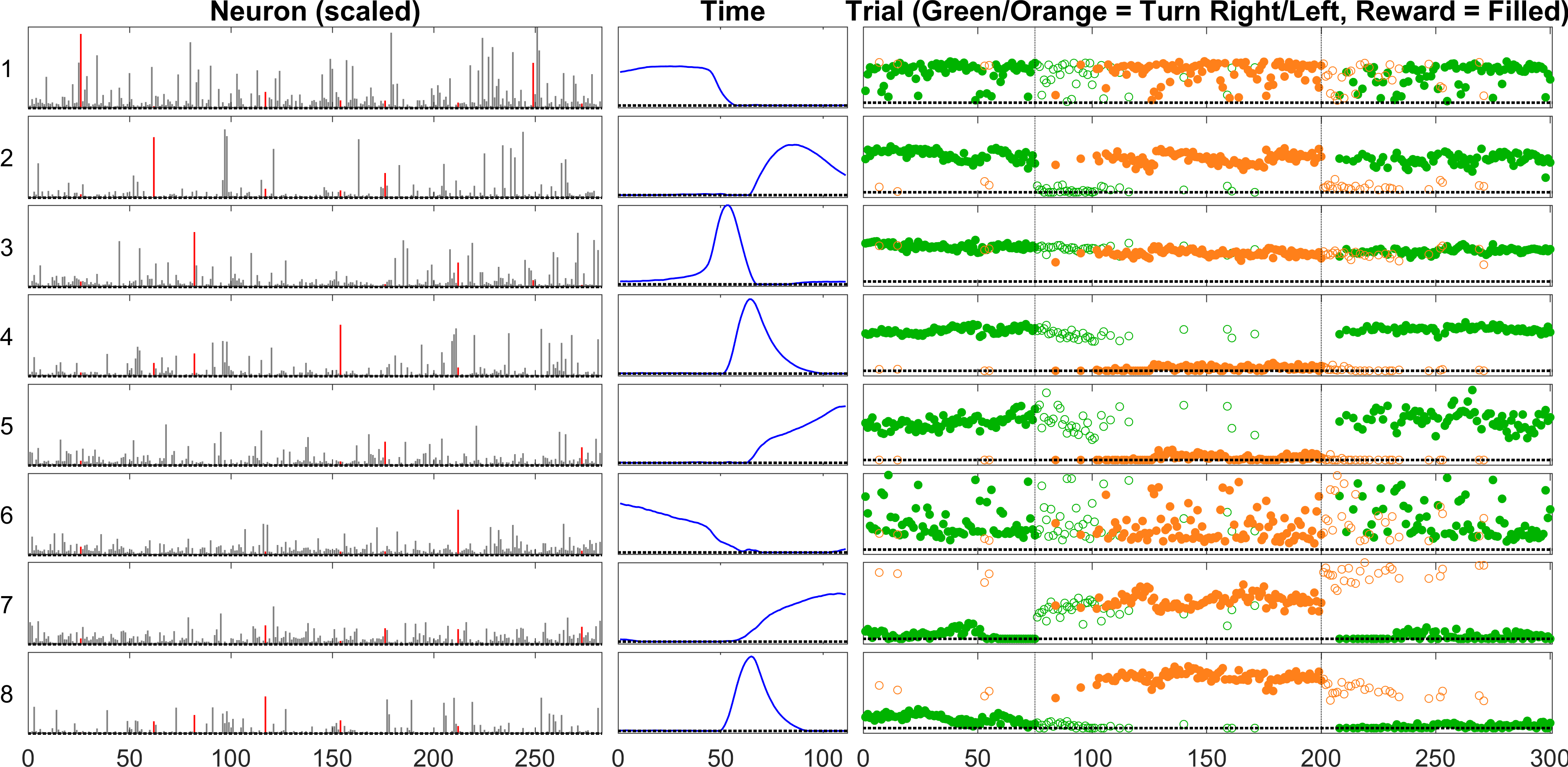}}
  \caption{Additional GCP tensor decompositions of mouse neural activity.}
  \label{fig:mouse_other}
\end{figure}

\subsubsection{Regression task for mouse neural activity}
\label{sec:regression-task}

Recall that the tensor factorization has no
knowledge of the experimental conditions, i.e., which way the mouse
turned or whether or not it received a reward.  Suppose that the
experimental logs were corrupted in such a way that we lost 50\% of
the trial indicators (completely at random rather than in a sequence). For instance, we might not know whether the mouse
turned left or right in Trial 87.  We can use the results of the
GCP tensor factorization to recover that information.
Observe that each trial is represented by 8 values, i.e., a score for each component.
These vectors can be used for regression.

Our experimental setup is as follows. We randomly selected 50\% of the 300
trials as training data and use the remainder for testing.
We do simple linear regression. Specifically, we let $\Ak{3}^{\rm train}$ be the rows of $\Ak{3}$ corresponding to the training trials
and $\mathbf{y}^{\rm train}$ be the corresponding binary responses (e.g., 1 for left turn and 0 for right turn). We solve the regression problem:
\begin{displaymath}
  \min_{\V{\beta}}
  \| \Ak{3}^{\rm train} \V{\beta} - \mathbf{y}^{\rm train} \|.
\end{displaymath}
We let $\Ak{3}^{\rm test}$ be the rows of $\Ak{3}$ corresponding to the testing trials.
Using the optimal $\V{\beta}$,  we make predictions for $\mathbf{y}^{\rm test}$ by computing
\begin{displaymath}
  \V[\hat]{y}^{\rm test} =  \left[ \Ak{3}^{\rm test} \V{\beta} \geq 0.5 \right].
\end{displaymath}
We did this 100 times, both for determining the turn direction (left or right) and the reward (yes or no).

The results are shown in \cref{tab:regression}. We caution that these are merely for illustrative purposes as changing the ranks and other
parameters might impact the relative performance of the methods.
For the turn results, shown in \cref{tab:regression_turn}, only the Gamma loss failed to achieve perfect classification.
We can see which factors were most important based on the regression coefficients. For instance, the sixth component is clearly the most important for Gaussian, whereas the fifth and seventh are key for $\beta$-divergence.
The reward was harder to predict, per the results in \cref{tab:regression_reward}.
This is likely due to the fact that there were relatively few times when the reward was not received.
For instance, the Rayleigh method performed worst, in contrast to its perfect classification for the turn direction.
Only the $\beta$-divergence achieved \emph{perfect} regression with the third component being the most important predictor.

\begin{table}
  \centering\footnotesize
  \subfloat[Turn]{\label{tab:regression_turn}
  \begin{tabular}{|c|r*7{@{\;\;\;}r}|c|c|} \hline
    Loss  & \multicolumn{8}{c|}{Regression Coefficients} & Max      & Incorrect\\
    Type & 1 & 2 & 3 & 4 & 5 & 6 & 7 & 8                 & Std.~Dev.& out of 15000\\
    \hline
Gaussian & -9.6 & 2.1 & 0.5 & -0.8 & 3.7 & 15.9 & 3.5 & 1.3 & 2.2e+00 & 0 \\ \hline
Beta Div. & 5.5 & 5.4 & -4.6 & 3.0 & 5.9 & -1.8 & -5.6 & 1.9 & 1.2e+00 & 0 \\ \hline
Rayleigh & 2.7 & 1.9 & 1.2 & 0.9 & 5.6 & 3.7 & -5.3 & -0.4 & 1.2e+00 & 0 \\ \hline
Gamma & -15.1 & 22.4 & 6.2 & 4.3 & -0.3 & -7.6 & -8.2 & 10.5 & 3.0e+00 & 1454 \\ \hline
Huber & 2.8 & -1.3 & 3.4 & 9.7 & -0.6 & 1.4 & -1.5 & -2.7 & 7.1e-01 & 0 \\ \hline
  \end{tabular}}\\
  \subfloat[Reward]{\label{tab:regression_reward}
  \begin{tabular}{|c|r*7{@{\;\;\;}r}|c|c|} \hline
    Loss  & \multicolumn{8}{c|}{Regression Coefficients} & Max      & Incorrect\\
    Type & 1 & 2 & 3 & 4 & 5 & 6 & 7 & 8                 & Std.~Dev.& out of 15000\\
    \hline
Gaussian & 11.6 & -0.5 & 18.7 & -2.1 & -6.9 & -8.6 & 0.0 & -3.2 & 3.6e+00 & 37 \\ \hline
Beta Div. & 5.1 & -0.8 & 7.4 & -0.1 & 2.8 & -3.8 & 2.6 & 2.4 & 1.1e+00 & 0 \\ \hline
Rayleigh & -6.3 & 8.5 & 8.1 & 1.0 & -1.6 & 5.1 & 1.9 & -3.0 & 1.3e+00 & 520 \\ \hline
Gamma & 10.7 & 1.9 & 0.5 & 0.3 & -2.1 & 3.6 & 5.6 & -6.4 & 1.3e+00 & 172 \\ \hline
Huber & 3.0 & 13.5 & -9.0 & 2.3 & 2.5 & 2.2 & -1.0 & 4.0 & 1.3e+00 & 62 \\ \hline
  \end{tabular}}
  \caption{Regression coefficients and prediction performance for different loss functions}
  \label{tab:regression}
\end{table}

\subsection{Rainfall in India}
\label{sec:rainfall-india}

We consider monthly rainfall data for different regions in India for the period 1901--2015,
available from Kaggle\footnote{https://www.kaggle.com/rajanand/rainfall-in-india}.
For each of 36 regions, 12 months, and 115 years, we have the total rainfall in millimeters.
There is a small amount of missing data (0.72\%), which GCP handles explicitly.
We show example monthly rainfalls for 6 regions in \cref{fig:rainfall_examples}.

\begin{figure}[tpb]
  \centering
  \includegraphics[width=0.9\textwidth]{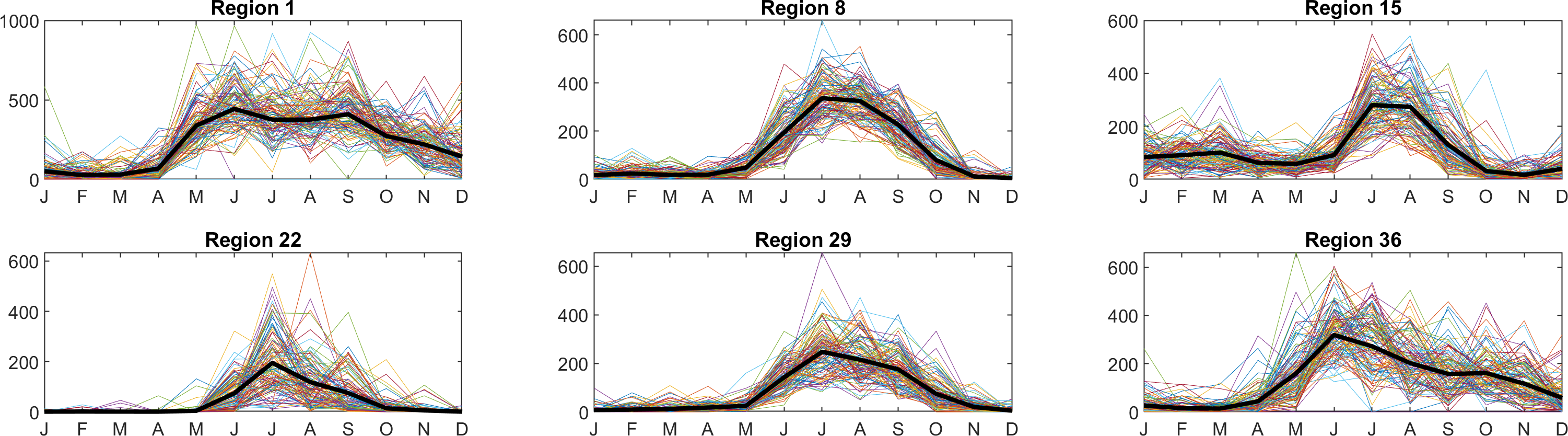}
  \caption{Rainfall totals per month in several regions in India. Each colored thin line represents a single year. The average is shown as a thick black line. Monsoon season is June -- September.}
  \label{fig:rainfall_examples}
\end{figure}

Oftentimes the gamma distribution is used to model rainfall.
A histogram of all monthly values is shown in \cref{fig:rainfall_hist} along with the estimated gamma distribution (in red), and
it seems as though a gamma distribution is potentially a reasonable model.
Most rainfall totals are very small (the smallest nonzero value is 0.1mm, which is presumably the precision of the measurements),
but the largest rainfall in a month exceeds 2300mm.
For this reason, we consider the GCP tensor decomposition with gamma loss.

\begin{figure}[tpb]
  \centering
  \includegraphics[width=0.7\textwidth]{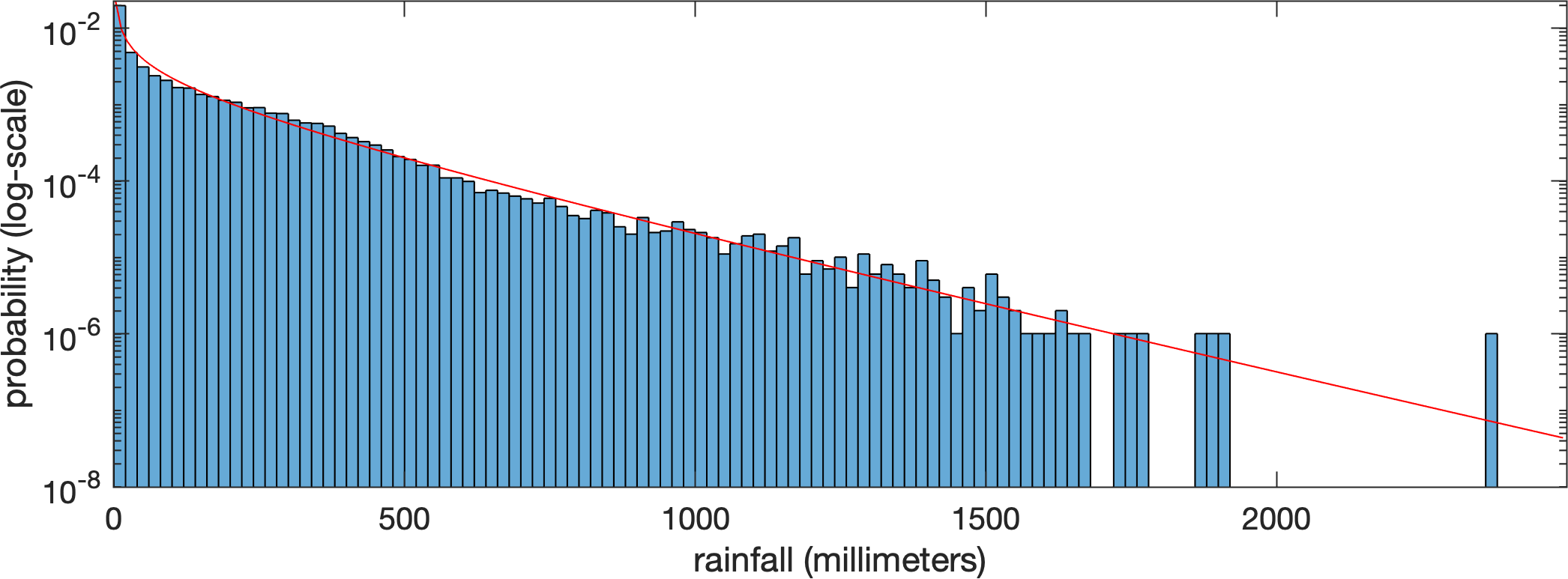}
  \caption{Histogram of monthly rainfall totals for 36 regions in India over 115 years. The estimated gamma distribution is shown in red.}
  \label{fig:rainfall_hist}
\end{figure}

A comparison of two GCP tensor decompositions is shown
in \cref{fig:rainfall}.
Factors in the first two modes (region and year) are normalized to length one, and the
monthly factor is normalized by the size of the component.
The rainfall from year to year follows no clear pattern, and this is consistent with
the general understanding of these rainfall patterns.
India is known for its monsoons, which occur in June--September of each year.

The GCP with standard Gaussian error loss and nonnegative constraints is shown in \cref{fig:rainfall_nn}.
The first component captures the period July--September, which is the main part of the monsoon season.
Components 3, 4, and 5 are dominated by a few regions. It is well known that Gaussian fitting can be swamped by
outliers, and this may be the case here.

The GCP with the gamma distribution loss function is shown in \cref{fig:rainfall_gamma}.
This captures the monsoon season primarily in the first two components. There are no particular
regions that dominate the factors.

\begin{figure}[tpb]
  \centering
  \subfloat[Gaussian (standard CP) with nonnegativity constraints, which separates July into its own component.]{\label{fig:rainfall_nn}%
    \includegraphics[width=\textwidth]{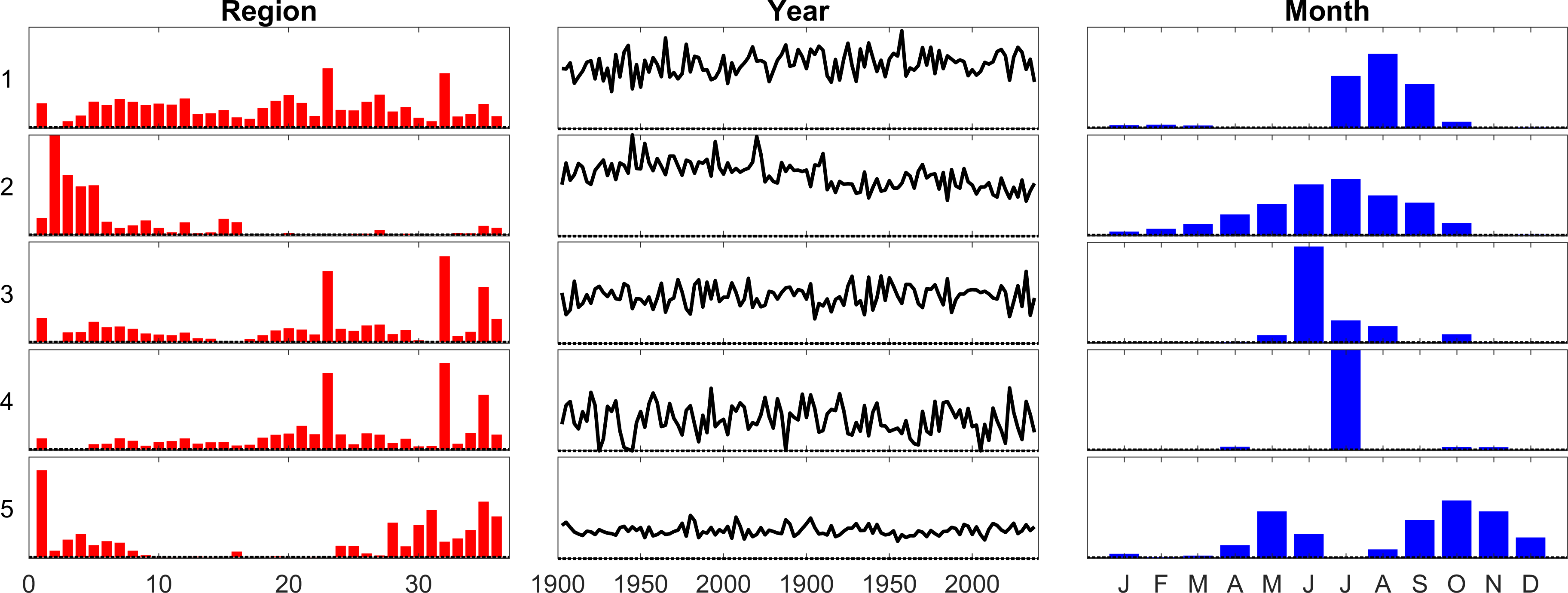}}\\
  \subfloat[Gamma (with nonnegativity constraints), which picks up the monsoon in the first component.]{\label{fig:rainfall_gamma}%
    \includegraphics[width=\textwidth]{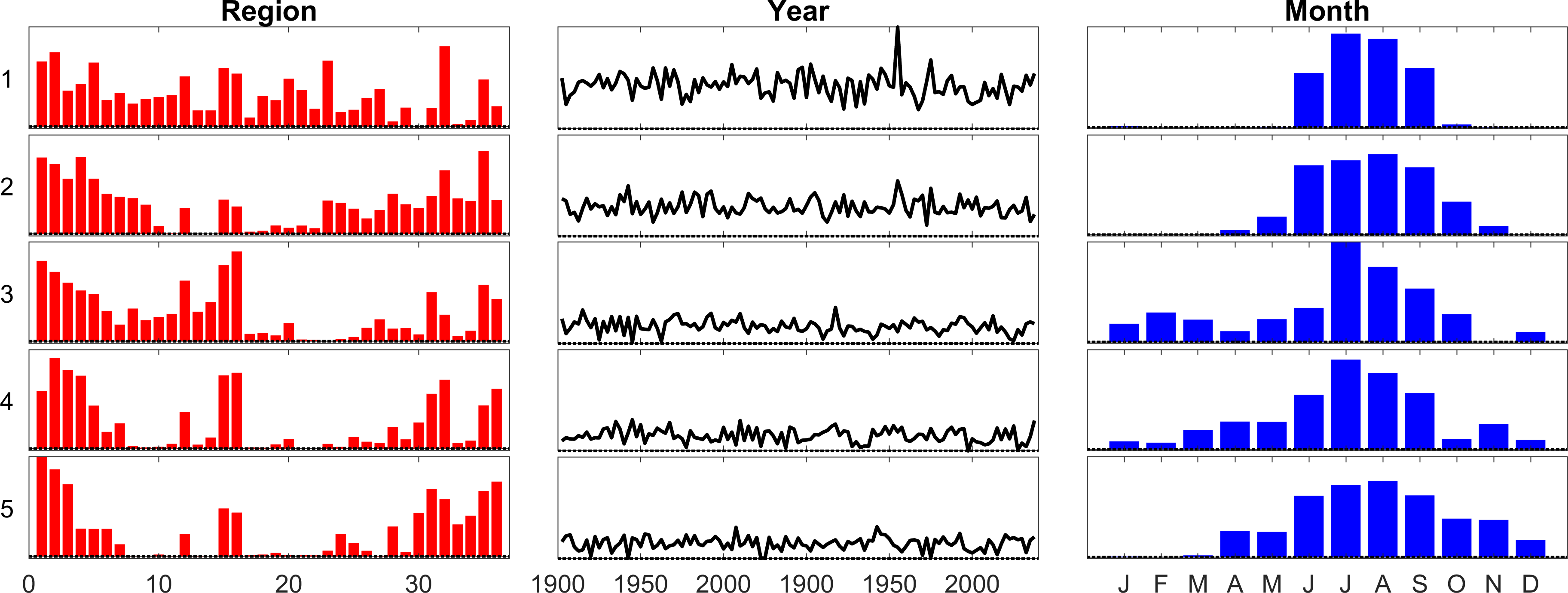}}
  \caption{GCP tensor decomposition of India rainfall data, organized into a tensor
    of 36 regions, 115 years, and 12 months. The first two modes are normalized to length 1}
  \label{fig:rainfall}
\end{figure}

\section{Conclusions and future work}
\label{sec:concl-future-work}

We have presented the GCP tensor decomposition framework which allows
the use of an arbitrary elementwise loss function, generalizing
previous works and enabling some extensions. GCP includes standard CP
tensor decomposition and Poisson tensor decomposition \cite{ChKo12}, as well as
decompositions based on beta divergences \cite{CiZdChPl07}.  Using the GCP framework, we
are able to define Bernoulli tensor decomposition for binary data,
which is something like the tensor decomposition version of logistic
regression and is derived via maximum likelihood.
Alternatively, GCP can also handle a heuristic loss function such as Huber loss.
We do not claim that any particular loss function is necessarily better than any other.
Rather,
for data analysis, it is often useful to have a variety of tools
available, and GCP provides flexibility in terms of choosing among different loss functions
to fit the needs of the analyst.
Additionally, the GCP framework \emph{efficiently} manages missing data, which is a common difficulty in practice.
Our main theorem (\cref{thm:grad}) generalizes prior results for the gradient
in the case of standard least squares, Poisson tensor factorization, and for missing data.
It further reveals that the gradient takes the form of an MTTKRP, enabling the use of efficient implementations for this key tensor operation.

In our framework, we have proposed that the weights $\wi$ be used as
indicators for missingness and restricted as $\wi \in
\set{0,1}$. To generalize this, we can easily incorporate nonnegative elementwise
weights $\wi \geq 0$.  For instance, we might give higher or lower
weights depending on the confidence in the data measurements.  In
recommender systems, there is also an idea that missing data may not be
entirely missing at random. In this case, it may be useful to treat
missing data elements as zeros but with low weights; see, e.g.,
\cite{St10a}.

For simplicity, our discussion focused on using the same elementwise loss function $f(\xi,\mi)$
for all entries of the tensor. However, we could easily define a different
loss function for every entry, i.e., $f_i(\xi,\mi)$.
The only modification is to the definition \cref{eq:Y} of the elementwise derivative tensor $\Y$.
If we have a heterogeneous mixture of data types, this may be appropriate.
In the matrix case, Udell, Horn, Zadeh, and Boyd \cite{UdHoZaBo16}
have proposed generalized low-rank models (GLRMs)
which use a different loss function for each column in matrix factorization.
We have also assumed our loss functions are continuously differentiable with respect to $\mi$,
but that can potentially be relaxed as well in the same way as done by Udell et al.~\cite{UdHoZaBo16}.

In our discussion of scarcity in \cref{sec:fitting-gcp-sparse}, we
alluded to the potential utility of imposing scarcity for scaling up
to larger scale tensors.  In stochastic gradient descent, for example, we impose
scarcity by selecting only a few elements of the tensor at each
iteration. Another option is to purposely omit most of the data,
depending on the inherent redundancies in the data (assuming it is sufficiently
incoherent).
These are topics
that we will investigate in detail in future work.

Lastly, it may also be of interest to  extend the GCP framework to functional tensor decomposition.
Garcke~\cite{Ga10}, e.g., has used hinge and Huber losses for fitting a functional
version of the CP tensor decomposition.

\appendix

\section{Kruskal tensors with explicit weights}%
\label{sec:explicit-weights}%
It is sometimes convenient to write \cref{eq:M-low-rank} with explicit positive weights $\lvec \in \Real_{+}^{r}$, i.e.,
\begin{equation}
  \label{eq:M-low-rank-lambda}
  \mi* = \sum_{j=1}^r \lambda(j) \; \Ake{1} \, \Ake{2} \cdots \Ake{d},
\end{equation}
with shorthand $\M = \KT*$.
In this case, the mode-$k$ unfolding in \cref{eq:Zk} is instead given by
\begin{displaymath}
  \Mk = \Ak \diag(\lvec) \Zk^T.
\end{displaymath}

We can also define the vectorized form
\begin{equation}
  \label{eq:Ms-unfold}
  \M = \KT*
   \Rightarrow
   \Mv = \zvec \bm{\lambda},
\end{equation}
where
\begin{equation}
  \label{eq:Z}
  \zvec \equiv \Ak{d} \odot \Ak{d-1} \odot \cdots \odot \Ak{1}
  \in \Real^{n^d \times r}.
\end{equation}

Using these definitions, it is a straightforward exercise to extend \cref{thm:grad} to the case $\M = \KT*$.

\begin{corollary}
  \label{cor:grad-alt}
  Let the conditions of \cref{thm:grad} hold except that the model has an explicit weight vector so that
  $\M = \KT*$. In this case,
  the partial derivatives of $F$ w.r.t.\@
  $\Ak$ and $\bm{\lambda}$ are
  \begin{equation}
    \FPD{F}{\Ak} =
    \Yk \Zk \diag(\bm{\lambda})
    \qtext{and}
    \FPD{F}{\bm{\lambda}} = \zvec^T \Yv,
  \end{equation}
  where $\Yk$ and $\Yv$ are, respectively, the mode-$k$ unfolding and
  vectorization of the tensor $\Y$ defined in \cref{eq:Y},
  $\Zk$ is defined in \cref{eq:Zk}, and $\zvec$ is defined in \cref{eq:Z}.
\end{corollary}

\section{Special structure of standard CP gradient}
\label{sec:spec-struct-stand}
In standard CP, which uses $f(x,m) = (x-m)^2$, the gradient has special structure that can be exploited when $\X$ is sparse.
Leaving out the constant, $\FPD{f}{m} = -x + m$; therefore, $\Y = -\X + \M$.
From \cref{eq:GCP-grad}, the CP gradient is
\begin{equation}\label{eq:cp-grad}
  \FPD{F}{\Ak} = -(\Xk - \Mk)\Zk = -\Xk\Zk + \Ak(\Zk'\Zk").
\end{equation}
The first term is an MTTKRP with the original tensor, and so it can exploit sparsity if $\X$ is sparse,
reducing the cost from $\mathcal{O}(rn^d)$ to $\mathcal{O}(r^2d \cdot \text{nnz}(\X))$ and avoiding forming $\Zk$ explicitly.
The second term can also avoid forming $\Zk$ explicitly since its gram matrix is given by
\begin{equation}\label{eq:Zk-gram}
  \Zk'\Zk" = (\Ak{1}'\Ak{1}") \ast \cdots \ast (\Ak{k-1}'\Ak{k-1}") \ast  (\Ak{k+1}'\Ak{k+1}") \ast \cdots \ast (\Ak{d}'\Ak{d}"),
\end{equation}
where $\ast$ is the Hadamard (elementwise) product. This means that $\Zk'\Zk"$ is trivial to compute,
requiring only $\mathcal{O}(r^2d \bar n)$ operations.
\Cref{eq:cp-grad} is a well-known result; see, e.g., \cite{AcDuKo11}.
Computation of MTTKRP with a sparse tensor is discussed further in \cite{BaKo07}.

\section{GCP optimization}
\label{sec:gcp-optimization}
First-order optimization methods expect a vector-valued function $f:\Real^n \rightarrow \Real$ and a corresponding vector-valued gradient,
but our variable is the set of $d$ factor matrices.
Because it may not be immediately obvious,
we briefly explain how to make the conversion.
We define the function \textsc{kt2vec} to convert a Kruskal tensor, i.e., a set of factor matrices, as follows:
\begin{displaymath}
  \avec \gets \textsc{kt2vec}(\Akset) \equiv \left[ \vc(\Ak{1});\, \vc(\Ak{2});\, \dots;\, \vc(\Ak{d}) \right].
\end{displaymath}
The \vc\@ operator converts a matrix to a column vector by stacking its columns, and we use MATLAB-like semicolon notation to say that the \textsc{kt2vec} operator stacks all those vectors on top of each other. We can define a corresponding inverse operator, \textsc{vec2kt}.
The number of variables in the set of factor matrices $\Akset$ is $d r \bar n$, and this is exactly the same number in the vector version $\avec$ because it is just a rearrangement of the entries in the factor matrices. Since the entries of the gradient matrices correspond to the same entries in the factor matrices, we use the same transformation function for them.
The wrapper that would be used to call an optimization method is shown in \cref{alg:wrapper}.
The optimization method would input a vector optimization variable, this is converted to a sequence of matrices, we compute the function and gradient using \cref{alg:gcp-fg}, we turn the gradients into a vector, and we return this along with the function value.

\begin{algorithm}
  \caption{Wrapper for using first-order optimization method}\label{alg:wrapper}
  \begin{algorithmic}[1]
    \Function{gcp\_fg\_wrapper}{\avec}
    \State $\Akset \gets \textsc{vec2kt}(\avec)$
    \State $[F,\Gkset] \gets \textsc{gcp\_fg}(\X,\Omega,\Akset)$
    \State $\gvec \gets \textsc{kt2vec}(\Gkset)$
    \State \Return $[F,\gvec]$
    \EndFunction
  \end{algorithmic}
\end{algorithm}

\section{Acknowledgments}%
\label{sec:acknowledgments}%
We would like to acknowledge the contributions of our colleague Cliff
Anderson-Bergman, who served as a mentor to David Hong in this work.
Thanks for Age Smilde for suggesting negative binomial, to Michiel
Vandecappelle for inspiring the discussion of beta divergences, and to
Martin Mohlenkamp for making us aware of the work by Garcke~\cite{Ga10}.
Thanks to Brett Larsen for feedback on earlier drafts of this work.
We are extremely grateful to the anonymous referees for critical feedback that led to significant improvements. 
\bibliographystyle{siamplainmod}

\begin{thebibliography}{10}

\bibitem{AcBiBiBr07}
{\sc E.~Acar, C.~A. Bingol, H.~Bingol, R.~Bro, and B.~Yener}, {\em Multiway
  analysis of epilepsy tensors}, Bioinformatics, 23 (2007), pp.~i10--i18,
  \href{http://dx.doi.org/10.1093/bioinformatics/btm210}
  {\nolinkurl{doi:10.1093/bioinformatics/btm210}}.

\bibitem{AcDuKo11}
{\sc E.~Acar, D.~M. Dunlavy, and T.~G. Kolda}, {\em A scalable optimization
  approach for fitting canonical tensor decompositions}, Journal of
  Chemometrics, 25 (2011), pp.~67--86,
  \href{http://dx.doi.org/10.1002/cem.1335} {\nolinkurl{doi:10.1002/cem.1335}}.

\bibitem{AcDuKoMo10}
{\sc E.~Acar, D.~M. Dunlavy, T.~G. Kolda, and M.~M{\o}rup}, {\em Scalable
  tensor factorizations with missing data}, in SDM10: Proceedings of the 2010
  SIAM International Conference on Data Mining, 2010, pp.~701--712,
  \href{http://dx.doi.org/10.1137/1.9781611972801.61}
  {\nolinkurl{doi:10.1137/1.9781611972801.61}}.

\bibitem{AcDuKoMo11}
{\sc E.~Acar, D.~M. Dunlavy, T.~G. Kolda, and M.~M{\o}rup}, {\em Scalable
  tensor factorizations for incomplete data}, Chemometrics and Intelligent
  Laboratory Systems, 106 (2011), pp.~41--56,
  \href{http://dx.doi.org/10.1016/j.chemolab.2010.08.004}
  {\nolinkurl{doi:10.1016/j.chemolab.2010.08.004}}.

\bibitem{AcYe09}
{\sc E.~Acar and B.~Yener}, {\em Unsupervised multiway data analysis: A
  literature survey}, IEEE Transactions on Knowledge and Data Engineering, 21
  (2009), pp.~6--20, \href{http://dx.doi.org/10.1109/TKDE.2008.112}
  {\nolinkurl{doi:10.1109/TKDE.2008.112}}.

\bibitem{BaKo06}
{\sc B.~W. Bader and T.~G. Kolda}, {\em Algorithm 862: {MATLAB} tensor classes
  for fast algorithm prototyping}, ACM Transactions on Mathematical Software,
  32 (2006), pp.~635--653, \href{http://dx.doi.org/10.1145/1186785.1186794}
  {\nolinkurl{doi:10.1145/1186785.1186794}}.

\bibitem{BaKo07}
{\sc B.~W. Bader and T.~G. Kolda}, {\em Efficient {MATLAB} computations with
  sparse and factored tensors}, SIAM Journal on Scientific Computing, 30
  (2007), pp.~205--231, \href{http://dx.doi.org/10.1137/060676489}
  {\nolinkurl{doi:10.1137/060676489}}.

\bibitem{TensorToolbox}
{\sc B.~W. Bader, T.~G. Kolda, et~al.}, {\em {MATLAB Tensor Toolbox Version
  3.0-dev}}.
\newblock Available online, Oct. 2017, \url{https://www.tensortoolbox.org}.

\bibitem{BaKnRo17}
{\sc G.~Ballard, N.~Knight, and K.~Rouse}, {\em Communication lower bounds for
  matricized tensor times {Khatri-Rao} product}, 2017,
  \href{http://arxiv.org/abs/1708.07401v1} {arXiv:1708.07401v1 [cs.DC]}.

\bibitem{BeGaMo09}
{\sc G.~Beylkin, J.~Garcke, and M.~J. Mohlenkamp}, {\em Multivariate regression
  and machine learning with sums of separable functions}, {SIAM} Journal on
  Scientific Computing, 31 (2009), pp.~1840--1857,
  \href{http://dx.doi.org/10.1137/070710524}
  {\nolinkurl{doi:10.1137/070710524}}.

\bibitem{BeMo02}
{\sc G.~Beylkin and M.~J. Mohlenkamp}, {\em Numerical operator calculus in
  higher dimensions}, Proceedings of the National Academy of Sciences, 99
  (2002), pp.~10246--10251, \href{http://dx.doi.org/10.1073/pnas.112329799}
  {\nolinkurl{doi:10.1073/pnas.112329799}}.

\bibitem{BeMo05}
{\sc G.~Beylkin and M.~J. Mohlenkamp}, {\em Algorithms for numerical analysis
  in high dimensions}, SIAM Journal on Scientific Computing, 26 (2005),
  pp.~2133--2159, \href{http://dx.doi.org/10.1137/040604959}
  {\nolinkurl{doi:10.1137/040604959}}.

\bibitem{Br97}
{\sc R.~Bro}, {\em {PARAFAC}. {Tutorial and applications}}, Chemometrics and
  Intelligent Laboratory Systems, 38 (1997), pp.~149--171,
  \href{http://dx.doi.org/10.1016/S0169-7439(97)00032-4}
  {\nolinkurl{doi:10.1016/S0169-7439(97)00032-4}}.

\bibitem{ByLuNoZh95}
{\sc R.~H. Byrd, P.~Lu, J.~Nocedal, and C.~Zhu}, {\em A limited memory
  algorithm for bound constrained optimization}, SIAM J. Sci. Comput., 16
  (1995), pp.~1190--1208, \href{http://dx.doi.org/10.1137/0916069}
  {\nolinkurl{doi:10.1137/0916069}}.

\bibitem{CaCh70}
{\sc J.~D. Carroll and J.~J. Chang}, {\em Analysis of individual differences in
  multidimensional scaling via an {N}-way generalization of ``{Eckart-Young}''
  decomposition}, Psychometrika, 35 (1970), pp.~283--319,
  \href{http://dx.doi.org/10.1007/BF02310791}
  {\nolinkurl{doi:10.1007/BF02310791}}.

\bibitem{ChKo12}
{\sc E.~C. Chi and T.~G. Kolda}, {\em On tensors, sparsity, and nonnegative
  factorizations}, SIAM Journal on Matrix Analysis and Applications, 33 (2012),
  pp.~1272--1299, \href{http://dx.doi.org/10.1137/110859063}
  {\nolinkurl{doi:10.1137/110859063}}.

\bibitem{CiAm10}
{\sc A.~Cichocki and S.~ichi Amari}, {\em Families of alpha- beta- and gamma-
  divergences: Flexible and robust measures of similarities}, Entropy, 12
  (2010), pp.~1532--1568, \href{http://dx.doi.org/10.3390/e12061532}
  {\nolinkurl{doi:10.3390/e12061532}}.

\bibitem{CiPh09}
{\sc A.~Cichocki and A.-H. Phan}, {\em Fast local algorithms for large scale
  nonnegative matrix and tensor factorizations}, IEICE Transactions on
  Fundamentals of Electronics, Communications and Computer Sciences, E92.A
  (2009), pp.~708--721, \href{http://dx.doi.org/10.1587/transfun.E92.A.708}
  {\nolinkurl{doi:10.1587/transfun.E92.A.708}}.

\bibitem{CiZdChPl07}
{\sc A.~Cichocki, R.~Zdunek, S.~Choi, R.~Plemmons, and S.-I. Amari}, {\em
  Non-negative tensor factorization using alpha and beta divergences}, in
  ICASSP 07: Proceedings of the International Conference on Acoustics, Speech,
  and Signal Processing, 2007,
  \href{http://dx.doi.org/10.1109/ICASSP.2007.367106}
  {\nolinkurl{doi:10.1109/ICASSP.2007.367106}}.

\bibitem{CoDaSc02}
{\sc M.~Collins, S.~Dasgupta, and R.~E. Schapire}, {\em A generalization of
  principal components analysis to the exponential family}, in NIPS'02:
  Advances in Neural Information Processing Systems 14, MIT Press, 2002,
  pp.~617--624,
  \url{http://papers.nips.cc/paper/2078-a-generalization-of-principal-components-analysis-to-the-exponential-family.pdf}.

\bibitem{CoLiKuGo15}
{\sc F.~Cong, Q.-H. Lin, L.-D. Kuang, X.-F. Gong, P.~Astikainen, and
  T.~Ristaniemi}, {\em Tensor decomposition of {EEG} signals: A brief review},
  Journal of Neuroscience Methods, 248 (2015), pp.~59--69,
  \href{http://dx.doi.org/10.1016/j.jneumeth.2015.03.018}
  {\nolinkurl{doi:10.1016/j.jneumeth.2015.03.018}}.

\bibitem{DuKoAc11}
{\sc D.~M. Dunlavy, T.~G. Kolda, and E.~Acar}, {\em Temporal link prediction
  using matrix and tensor factorizations}, ACM Transactions on Knowledge
  Discovery from Data, 5 (2011), p.~10 (27 pages),
  \href{http://dx.doi.org/10.1145/1921632.1921636}
  {\nolinkurl{doi:10.1145/1921632.1921636}}.

\bibitem{FaHeLi17}
{\sc L.~Fang, N.~He, and H.~Lin}, {\em {CP} tensor-based compression of
  hyperspectral images}, Journal of the Optical Society of America A, 34
  (2017), pp.~252--258, \href{http://dx.doi.org/10.1364/josaa.34.000252}
  {\nolinkurl{doi:10.1364/josaa.34.000252}}.

\bibitem{FeId11}
{\sc C.~F\'evotte and J.~Idier}, {\em Algorithms for nonnegative matrix
  factorization with the $\beta$-divergence}, Neural Computation, 23 (2011),
  pp.~2421--2456, \href{http://dx.doi.org/10.1162/NECO_a_00168}
  {\nolinkurl{doi:10.1162/NECO_a_00168}}.

\bibitem{Ga10}
{\sc J.~Garcke}, {\em Classification with sums of separable functions}, in
  Machine Learning and Knowledge Discovery in Databases, J.~L. Balc{\'a}zar,
  F.~Bonchi, A.~Gionis, and M.~Sebag, eds., Springer Berlin Heidelberg, 2010,
  pp.~458--473, \href{http://dx.doi.org/10.1007/978-3-642-15880-3_35}
  {\nolinkurl{doi:10.1007/978-3-642-15880-3_35}}.

\bibitem{Go03}
{\sc G.~J. Gordon}, {\em Generalized$^2$ linear$^2$ models}, in Advances in
  Neural Information Processing Systems (NIPS'03), 2003, pp.~593--600.

\bibitem{HaKh07}
{\sc W.~Hackbusch and B.~N. Khoromskij}, {\em Tensor-product approximation to
  operators and functions in high dimensions}, Journal of Complexity, 23
  (2007), pp.~697--714, \href{http://dx.doi.org/10.1016/j.jco.2007.03.007}
  {\nolinkurl{doi:10.1016/j.jco.2007.03.007}}.

\bibitem{HaPlKo15}
{\sc S.~Hansen, T.~Plantenga, and T.~G. Kolda}, {\em Newton-based optimization
  for {Kullback-Leibler} nonnegative tensor factorizations}, Optimization
  Methods and Software, 30 (2015), pp.~1002--1029,
  \href{http://dx.doi.org/10.1080/10556788.2015.1009977}
  {\nolinkurl{doi:10.1080/10556788.2015.1009977}},
  \href{http://arxiv.org/abs/1304.4964} {arXiv:1304.4964}.

\bibitem{Ha70}
{\sc R.~A. Harshman}, {\em Foundations of the {PARAFAC} procedure: Models and
  conditions for an ``explanatory" multi-modal factor analysis}, UCLA working
  papers in phonetics, 16 (1970), pp.~1--84.
\newblock Available at
  \url{http://www.psychology.uwo.ca/faculty/harshman/wpppfac0.pdf}.

\bibitem{HaTiFr09}
{\sc T.~Hastie, R.~Tibshrirani, and J.~Friedman}, {\em The Elements of
  Statitical Learning}, Springer, 2nd~ed., 2009,
  \href{http://dx.doi.org/10.1007/978-0-387-84858-7}
  {\nolinkurl{doi:10.1007/978-0-387-84858-7}}.

\bibitem{HaBaJiTo17}
{\sc K.~Hayashi, G.~Ballard, J.~Jiang, and M.~Tobia}, {\em Shared memory
  parallelization of {MTTKRP} for dense tensors}, 2017,
  \href{http://arxiv.org/abs/1708.08976v1} {arXiv:1708.08976v1 [cs.DC]}.

\bibitem{Hi27a}
{\sc F.~L. Hitchcock}, {\em The expression of a tensor or a polyadic as a sum
  of products}, Journal of Mathematics and Physics, 6 (1927), pp.~164--189,
  \href{http://dx.doi.org/10.1002/sapm192761164}
  {\nolinkurl{doi:10.1002/sapm192761164}}.

\bibitem{HuSiLi15}
{\sc K.~Huang, N.~D. Sidiropoulos, and A.~P. Liavas}, {\em A flexible and
  efficient algorithmic framework for constrained matrix and tensor
  factorization}, IEEE Transactions on Signal Processing, 64 (2016),
  pp.~5052--5065, \href{http://dx.doi.org/10.1109/TSP.2016.2576427}
  {\nolinkurl{doi:10.1109/TSP.2016.2576427}},
  \href{http://arxiv.org/abs/1506.04209v2} {arXiv:1506.04209v2}.

\bibitem{Hu64}
{\sc P.~J. Huber}, {\em Robust estimation of a location parameter}, Annals of
  Statistics, 53 (1964), pp.~73--101,
  \href{http://dx.doi.org/10.1214/aoms/1177703732}
  {\nolinkurl{doi:10.1214/aoms/1177703732}}.

\bibitem{JaCaYa14}
{\sc R.~Jaff{\'{e}}, K.~M. Cawley, and Y.~Yamashita}, {\em Applications of
  excitation emission matrix fluorescence with parallel factor analysis
  ({EEM}-{PARAFAC}) in assessing environmental dynamics of natural dissolved
  organic matter ({DOM}) in aquatic environments: A review}, in Advances in the
  Physicochemical Characterization of Dissolved Organic Matter: Impact on
  Natural and Engineered Systems, vol.~1160 of ACS Symposium Series, American
  Chemical Society ({ACS}), 2014, pp.~27--73,
  \href{http://dx.doi.org/10.1021/bk-2014-1160.ch003}
  {\nolinkurl{doi:10.1021/bk-2014-1160.ch003}}.

\bibitem{KoBa06}
{\sc T.~Kolda and B.~Bader}, {\em The {TOPHITS} model for higher-order web link
  analysis}, in Proceedings of Link Analysis, Counterterrorism and Security
  2006, 2006, \url{http://www.siam.org/meetings/sdm06/workproceed/Link
  Analysis/21Tamara_Kolda_SIAMLACS.pdf}.

\bibitem{Ko17}
{\sc T.~G. Kolda}, {\em Sparse versus scarce}.
\newblock Blog, Nov. 2017,
  \url{http://www.kolda.net/post/sparse-versus-scarce/}.
\newblock Accessed May 2018.

\bibitem{KoBa09}
{\sc T.~G. Kolda and B.~W. Bader}, {\em Tensor decompositions and
  applications}, SIAM Review, 51 (2009), pp.~455--500,
  \href{http://dx.doi.org/10.1137/07070111X}
  {\nolinkurl{doi:10.1137/07070111X}}.

\bibitem{LeSe99}
{\sc D.~D. Lee and H.~S. Seung}, {\em Learning the parts of objects by
  non-negative matrix factorization}, Nature, 401 (1999), pp.~788--791,
  \href{http://dx.doi.org/10.1038/44565} {\nolinkurl{doi:10.1038/44565}}.

\bibitem{LiMaYaVu16}
{\sc J.~Li, Y.~Ma, C.~Yan, and R.~Vuduc}, {\em Optimizing sparse tensor times
  matrix on multi-core and many-core architectures}, in Workshop on Irregular
  Applications: Architecture and Algorithms (IA3), 2016, pp.~26--33,
  \href{http://dx.doi.org/10.1109/IA3.2016.010}
  {\nolinkurl{doi:10.1109/IA3.2016.010}}.

\bibitem{MaGuFa11}
{\sc K.~Maruhashi, F.~Guo, and C.~Faloutsos}, {\em {MultiAspectForensics}:
  Pattern mining on large-scale heterogeneous networks with tensor analysis},
  in 2011 International Conference on Advances in Social Networks Analysis and
  Mining (ASONAM), IEEE, 2011, pp.~203--210,
  \href{http://dx.doi.org/10.1109/asonam.2011.80}
  {\nolinkurl{doi:10.1109/asonam.2011.80}}.

\bibitem{MuStGrBr13}
{\sc K.~R. Murphy, C.~A. Stedmon, D.~Graeber, and R.~Bro}, {\em Fluorescence
  spectroscopy and multi-way techniques. {PARAFAC}}, Analytical Methods, 5
  (2013), p.~6557, \href{http://dx.doi.org/10.1039/c3ay41160e}
  {\nolinkurl{doi:10.1039/c3ay41160e}}.

\bibitem{NiTr13}
{\sc M.~Nickel and V.~Tresp}, {\em Logistic tensor factorization for
  multi-relational data}, 2013, \href{http://arxiv.org/abs/1306.2084v1}
  {arXiv:1306.2084v1 [stat.ML]}.

\bibitem{SNET}
{\sc T.~Opsahl}, {\em Network 1: Facebook-like social network},
  \url{https://toreopsahl.com/datasets/#online_social_network} (accessed
  2018-06-20).

\bibitem{OpPa09}
{\sc T.~Opsahl and P.~Panzarasa}, {\em Clustering in weighted networks}, Social
  Networks, 31 (2009), pp.~155--163,
  \href{http://dx.doi.org/10.1016/j.socnet.2009.02.002}
  {\nolinkurl{doi:10.1016/j.socnet.2009.02.002}},
  \url{https://doi.org/10.1016/j.socnet.2009.02.002}.

\bibitem{Pa97b}
{\sc P.~Paatero}, {\em Least squares formulation of robust non-negative factor
  analysis}, Chemometrics and Intelligent Laboratory Systems, 37 (1997),
  pp.~23--35, \href{http://dx.doi.org/10.1016/S0169-7439(96)00044-5}
  {\nolinkurl{doi:10.1016/S0169-7439(96)00044-5}}.

\bibitem{Pa97}
{\sc P.~Paatero}, {\em A weighted non-negative least squares algorithm for
  three-way ``{PARAFAC}'' factor analysis}, Chemometrics and Intelligent
  Laboratory Systems, 38 (1997), pp.~223--242,
  \href{http://dx.doi.org/10.1016/S0169-7439(97)00031-2}
  {\nolinkurl{doi:10.1016/S0169-7439(97)00031-2}}.

\bibitem{PaTa94}
{\sc P.~Paatero and U.~Tapper}, {\em Positive matrix factorization: A
  non-negative factor model with optimal utilization of error estimates of data
  values}, Environmetrics, 5 (1994), pp.~111--126,
  \href{http://dx.doi.org/10.1002/env.3170050203}
  {\nolinkurl{doi:10.1002/env.3170050203}}.

\bibitem{PaOpCa09}
{\sc P.~Panzarasa, T.~Opsahl, and K.~M. Carley}, {\em Patterns and dynamics of
  users{\textquotesingle} behavior and interaction: Network analysis of an
  online community}, Journal of the American Society for Information Science
  and Technology, 60 (2009), pp.~911--932,
  \href{http://dx.doi.org/10.1002/asi.21015}
  {\nolinkurl{doi:10.1002/asi.21015}}, \url{https://doi.org/10.1002/asi.21015}.

\bibitem{PaKaFaSi13}
{\sc E.~E. Papalexakis, U.~Kang, C.~Faloutsos, N.~D. Sidiropoulos, and
  A.~Harpale}, {\em Large scale tensor decompositions: Algorithmic developments
  and applications}, {IEEE} Data Eng. Bull., 36 (2013), pp.~59--66,
  \url{http://sites.computer.org/debull/A13sept/p59.pdf}.

\bibitem{PhTiCi13}
{\sc A.-H. Phan, P.~Tichavsky, and A.~Cichocki}, {\em Fast alternating {LS}
  algorithms for high order {CANDECOMP/PARAFAC} tensor factorizations}, IEEE
  Transactions on Signal Processing, 61 (2013), pp.~4834--4846,
  \href{http://dx.doi.org/10.1109/TSP.2013.2269903}
  {\nolinkurl{doi:10.1109/TSP.2013.2269903}},
  \url{http://dx.doi.org/10.1109/TSP.2013.2269903}.

\bibitem{PhTiCi13a}
{\sc A.-H. Phan, P.~Tichavsk\'{y}, and A.~Cichocki}, {\em Low complexity damped
  {Gauss}--{Newton} algorithms for {CANDECOMP/PARAFAC}}, SIAM Journal on Matrix
  Analysis and Applications, 34 (2013), pp.~126--147,
  \href{http://dx.doi.org/10.1137/100808034}
  {\nolinkurl{doi:10.1137/100808034}},
  \url{http://dx.doi.org/10.1137/100808034}.

\bibitem{Re12}
{\sc S.~Rendle}, {\em Factorization machines with {libFM}}, ACM Transactions on
  Intelligent Systems and Technology (TIST), 3 (2012), 57 (22~pages),
  pp.~57:1--57:22, \href{http://dx.doi.org/10.1145/2168752.2168771}
  {\nolinkurl{doi:10.1145/2168752.2168771}},
  \url{http://doi.acm.org/10.1145/2168752.2168771}.

\bibitem{ReDoBe16}
{\sc M.~J. Reynolds, A.~Doostan, and G.~Beylkin}, {\em Randomized alternating
  least squares for canonical tensor decompositions: Application to a {PDE}
  with random data}, SIAM Journal on Scientific Computing, 38 (2016),
  pp.~A2634--A2664, \href{http://dx.doi.org/10.1137/15M1042802}
  {\nolinkurl{doi:10.1137/15M1042802}}.

\bibitem{ShHa05}
{\sc A.~Shashua and T.~Hazan}, {\em Non-negative tensor factorization with
  applications to statistics and computer vision}, in ICML 2005: Proceedings of
  the 22nd International Conference on Machine Learning, 2005, pp.~792--799,
  \href{http://dx.doi.org/10.1145/1102351.1102451}
  {\nolinkurl{doi:10.1145/1102351.1102451}}.

\bibitem{SmKa15}
{\sc S.~Smith and G.~Karypis}, {\em Tensor-matrix products with a compressed
  sparse tensor}, in Proceedings of the 5th Workshop on Irregular Applications:
  Architectures and Algorithms, ACM, 2015, p.~7.

\bibitem{St10a}
{\sc H.~Steck}, {\em Training and testing of recommender systems on data
  missing not at random}, in KDD'10: Proceedings of the 16th ACM SIGKDD
  International Conference on Knowledge Discovery and Data Mining, 2010,
  pp.~713--722, \href{http://dx.doi.org/10.1145/1835804.1835895}
  {\nolinkurl{doi:10.1145/1835804.1835895}}.

\bibitem{UdHoZaBo16}
{\sc M.~Udell, C.~Horn, R.~Zadeh, and S.~Boyd}, {\em Generalized low rank
  models}, {FNT} in Machine Learning, 9 (2016), pp.~1--118,
  \href{http://dx.doi.org/10.1561/2200000055}
  {\nolinkurl{doi:10.1561/2200000055}},
  \url{http://dx.doi.org/10.1561/2200000055}.

\bibitem{WeWe01}
{\sc M.~Welling and M.~Weber}, {\em Positive tensor factorization}, Pattern
  Recognition Letters, 22 (2001), pp.~1255--1261,
  \href{http://dx.doi.org/10.1016/S0167-8655(01)00070-8}
  {\nolinkurl{doi:10.1016/S0167-8655(01)00070-8}}.

\bibitem{WiKiWaVy18}
{\sc A.~H. Williams, T.~H. Kim, F.~Wang, S.~Vyas, S.~I. Ryu, K.~V. Shenoy,
  M.~Schnitzer, T.~G. Kolda, and S.~Ganguli}, {\em Unsupervised discovery of
  demixed, low-dimensional neural dynamics across multiple timescales through
  tensor components analysis}, Neuron, 98 (2018), pp.~1099--1115,
  \href{http://dx.doi.org/10.1016/j.neuron.2018.05.015}
  {\nolinkurl{doi:10.1016/j.neuron.2018.05.015}}.

\bibitem{Yo41}
{\sc G.~Young}, {\em Maximum likelihood estimation and factor analysis},
  Psychometrika, 6 (1941), pp.~49--53,
  \href{http://dx.doi.org/10.1007/bf02288574}
  {\nolinkurl{doi:10.1007/bf02288574}}.

\bibitem{ZhWaPlPa08}
{\sc Q.~Zhang, H.~Wang, R.~J. Plemmons, and V.~P. Pauca}, {\em Tensor methods
  for hyperspectral data analysis: a space object material identification
  study}, J. Opt. Soc. Am. A, 25 (2008), pp.~3001--3012,
  \href{http://dx.doi.org/10.1364/JOSAA.25.003001}
  {\nolinkurl{doi:10.1364/JOSAA.25.003001}},
  \url{http://josaa.osa.org/abstract.cfm?URI=josaa-25-12-3001}.

\end{thebibliography}


\end{document}

%
%
%
%
